\documentclass[12pt]{amsart}
\usepackage{amssymb,amsfonts,amsmath,amsopn,amstext,amscd,color,latexsym,leqno,mathrsfs,url,xy}
\usepackage{extraipa}
\usepackage{tipa}
\usepackage{undertilde}
\usepackage{ushort}
\input xy
\xyoption{all}

\theoremstyle{remark}

\newtheorem{para}{\bf}[section]

\theoremstyle{definition}

\theoremstyle{plain}

\newtheorem{thm}[para]{Theorem}
\newtheorem{lemma}[para]{Lemma}

\newtheorem{cor}[para]{Corollary}
\newtheorem{prop}[para]{Proposition}

\newenvironment{numequation}
{\addtocounter{para}{1}\begin{equation}}{\end{equation}}

\newcommand{\bbF}{{\mathbb F}}

\newcommand{\bbN}{{\mathbb N}}

\newcommand{\bbQ}{{\mathbb Q}}
\newcommand{\bbR}{{\mathbb R}}

\newcommand{\bbZ}{{\mathbb Z}}

\newcommand{\bT}{{\bf T}}

\newcommand{\br}{{\bf r}}

\newcommand{\cA}{{\mathcal A}}

\newcommand{\cL}{{\mathcal L}}

\newcommand{\sM}{{\mathscr M}}

\newcommand{\tA}{\tilde{A}}

\newcommand{\tgA}{\tilde{A}_\bullet}

\newcommand{\Aleq}{A_{\leq\gamma}}

\newcommand{\Aeq}{A_{<\gamma}}

\newcommand{\tgB}{\tilde{B}_\bullet}\newcommand{\tgellB}{(\tilde{B_\ell})_\bullet}

\newcommand{\tk}{\tilde{k}}\newcommand{\tgk}{\tilde{k}_\bullet}
\newcommand{\tell}{\tilde{\ell}}\newcommand{\tgell}{\tilde{\ell}_\bullet}

\newcommand{\ta}{\tilde{a}}

\newcommand{\Aut}{{\rm Aut}}
\newcommand{\Cl}{{\rm Cl}}

\newcommand{\Frac}{{\rm Frac}}
\newcommand{\Gal}{{\rm Gal}}

\newcommand{\Hom}{{\rm Hom}}

\newcommand{\Pic}{{\rm{Pic}}}

\newcommand{\Spec}{{\rm Spec}}

\newcommand{\dv}{{\rm div}}

\newcommand{\car}{\stackrel{\cong}{\longrightarrow}}

\newcommand{\De}{\partial}

\newcommand{\aut}{{\rm Aut}\hskip3pt B_S}
\newcommand{\autl}{{\rm Aut}\hskip3pt B_{\ell}}
\newcommand{\autgrl}{{\rm Aut}^{\rm gr}\tgellB}
\newcommand{\autgrlo}{{\rm Aut}^{\rm gr}B_{\ell}}

\newcommand{\vg}{|\ell^\times|/|k^\times|}

\newcommand{\ellc}{\ell^{\circ}}\newcommand{\ellcc}{\ell^{\circ\circ}}

\newcommand{\tn}{\langle n\rangle}

\newcommand{\RA}{R(A)_\bullet}

\newcommand {\AUT}{{\cA}}

\begin{document}

\title{Forms of an affinoid disc and ramification}
\author{Tobias Schmidt}
\address{Institut f\"ur Mathematik, Humboldt-Universit\"at zu Berlin, Rudower Chaussee 25, D-12489 Berlin, Germany}
\email{ Tobias.Schmidt@math.hu-berlin.de}
\date{}
\maketitle

\scriptsize
\normalsize

\normalsize

\begin{abstract}
Let $k$ be a complete nonarchimedean field and let $X$ be an
affinoid closed disc over $k$. We classify the tamely ramified
twisted forms of $X$. Generalizing classical work of P. Russell on
inseparable forms of the affine line we construct explicit
families of wildly ramified forms of $X$. We finally compute the
class group and the Grothendieck group of forms of $X$ in certain
cases.

\end{abstract}


\vskip30pt

\section{Introduction}

\vskip8pt

Let $k$ be a field complete with respect to a nontrivial
nonarchimedean absolute value and let $X=X(r)$ be an affinoid
closed disc over $k$ of some real radius $r>0$. A form of $X$ is
an isomorphism class of a $k$-affinoid space that becomes
isomorphic to $X$ over some complete field extension
$k\subseteq\ell$. In this paper we prove some classification
results for such forms. Let $\AUT$ be the automorphism functor of
$X$.

\vskip8pt

Let $k\subseteq\ell$ be an extension and let $H^1(\ell/k,\AUT)$ be
the corresponding pointed set of forms of $X$. The behavior of the
latter set depends crucially on the ramification properties of the
extension $k\subseteq\ell$. To simplify things in this
introduction, let us assume for a moment that this extension is
finite and Galois. If $k\subseteq k^{\rm tr}\subseteq\ell$ denotes
the maximal tamely ramified subextension we have a short exact
sequence
$$ 1\longrightarrow  H^1(k^{\rm tr}/k, \AUT) \longrightarrow
H^1(\ell/k, \AUT)\longrightarrow H^1(\ell/k^{\rm tr},
\mathcal{A}_{k^{\rm tr}}).$$ 
Here $\mathcal{A}_{k^{\rm tr}}$ denotes the automorphism functor
of the base change $X\otimes k^{\rm tr}$. We propose to study the
two outer terms in the sequence. So suppose that $k\subseteq\ell$
is tamely ramified. We establish a canonical bijection of pointed
sets
$$ |\ell^{\times}|/|k^\times|\car H^1(\ell/k, \AUT),\hskip10pt {\rm class ~of~ |a|}\mapsto {\rm class~ of~}
X(r|a|)$$ where $X(r|a|)$ denotes the closed disc over $k$ of
radius $r|a|$. In particular, there are no unramified forms of
$X$. These results are in accordance with the corresponding
results for the open (poly-)disc obtained by A. Ducros
\cite{Ducros}. As in loc.cit. our proof depends on the theory of
graded reduction as introduced into analytic geometry by M. Temkin
\cite{TemkinII}. For example, any affinoid $k$-algebra $A$
leads, in a functorial way, to a $\bbR_+^\times$-graded algebra
$\tgA$ over the graded field
$\tgk$ (a graded field is a graded
ring in which any nonzero homogeneous element is invertible). It is induced by the spectral semi norm filtration on $A$.
In particular, the homogeneous degree $1$ part $\tA_1$ equals
the usual reduction of $A$. 
These graded rings are not local and so the local arguments of
\cite{Ducros} do not apply in our affinoid situation. In turn, we
only deal with $1$-dimensional spaces which makes the functor
$\cA$ more accessible to computation. Since $k\subseteq\ell$ is
tamely ramified, the extension between graded fields
$\tgk\subseteq\tgell$ is Galois with Galois group isomorphic to
$\Gal(\ell/k)$. We then use Galois descent
properties of the automorphism functors of $X$ and its graded
reduction to establish the result. For more information we refer
to the main body of the text.


\vskip8pt

We turn to the case of wild ramification. Here, we do not obtain a
complete picture. As a starting point, if $k\subseteq\ell$ is a
wildly ramified extension, its graded reduction
$\tgk\subseteq\tgell$ is purely inseparable. We show that the
classical construction of purely inseparable forms of the additive
group by P. Russell \cite{Russell} has a version over graded
fields. Each of these graded forms can then be lifted to a form of
the additive group on the analytic space $X$. This provides an
abundance of forms of $X$. Following \cite{KMT} we tentatively
call these forms {\it of Russell type}. Let us give more details
in the simplest case where $X$ is actually the unit disc. Let
$p=\;{\rm char}~\tk_1>0$ and $\tgk^{\rm alg}$ denote a graded
algebraic closure of the graded field $\tgk$. Let $k\subseteq\ell$
be an extension such that $(\tgk)^{p^{-n}}\subseteq\tgell$ for
some $n\geq 1$. Let $\tk_1[F]$ be the endomorphism ring of the
additive group over $\tk_1$ where $F$ is the Frobenius morphism.
The image $U_n$ in the quotient $\tk_1[F]/(F^n)$ of the
multiplicative monoid of separable endomorphisms is a group. Let
$G_n=U_n\times \tk_1^\times$ be the direct product. In
\cite{Russell} the author constructs an explicit action of $G_n$
on the pointed set $U_n$. We deduce a canonical injection of
pointed sets
$$U_n/G_n\hookrightarrow H^1(\ell/k, \AUT)$$
given by mapping the residue class mod $(F^n)$ of a separable
endomorphism $\sum_{i=1,...,m} a_iF^{i}$ to the isomorphism class
of the closed subgroup of the two dimensional additive group $X^2$
cut out by the equation
$$ T_2^{p^n}=a_0T_1+a_1T_1^p+\cdot\cdot\cdot +a_mT_1^{p^{m}}.$$
Here, $T_i$ are two parameters on $X^2$. For example, $U_1/G_1$
equals the quotient of a certain $\tk_1^\times$-action on an
infinite direct sum of copies of the space $\tk_1/\tk_1^p$. For
more details we refer to loc.cit. and the main body of our text.

\vskip8pt

The forms of the additive group on $X$ of Russell type have
geometrically reduced graded reduction, i.e. $\tgA\otimes
\tgk^{\rm alg}$ is reduced where $A$ denotes the affinoid algebra
of the form. In loc.cit. it is shown that there are many
inseparable forms of the affine line that fail to have a group
structure. In this light it is likely that, dropping the group
structure, there are many more wildly ramified forms of the space
$X$ with geometrically reduced graded reduction. In the final part
of our article we are concerned with basic invariants of such
forms such as the Picard group and the Grothendieck group. We work
under the assumptions that $k$ is discretely valued and that all
affinoid spaces are strict. These restrictions are technicalities
and should not be essential in the end. Let $Y$ be a wildly
ramified form of $X$ with geometrically reduced graded reduction.
Let $A$ be its affinoid algebra. Serre's theorem from algebraic
$K$-theory \cite{Bass} implies
$K_0(A)=\bbZ\oplus \Pic(A)$ (and this holds for {\it any}
form of $X$ regardless of ramification and reduction properties). We then use a
version of the $K_0$-part of Quillen's theorem \cite{LVOVB} to obtain a canonical isomorphism $\Pic(A)\simeq
\Pic(\tgA).$ Here, $\tgA$ is viewed as an abstract ring, i.e. we
forget the gradation here and in the following. Since $\tgA$ is
geometrically reduced, it equals a form of the affine line
relative to the ring extension $\tgk\subseteq\tgell$. Let $p^n$ be
the degree of the latter (finite free) extension. A choice of
uniformizer $\varpi\in \ell$ induces identifications
$$ \tgk=\tk_1[t^{\pm p^n}] \hskip10pt \subseteq \hskip10pt \tgell=\tk_1[t^{\pm 1}]$$
with rings of Laurent polynomials. We then have the classical
standard higher derivation associated with this $p$-radical
extension \cite{Jacobson}. To allow for $n>1$ we build on K. Baba's generalization to
higher exponents \cite{Baba} of P. Samuel's classical $p$-radical descent theory \cite{Samuel2} and give a fairly explicit
description of $\Pic(\tgA)$ in terms of logarithmic derivatives. We
deduce that the abelian group $\Pic(A)$ always has exponent $p^n$.
We also deduce a criterion when $\Pic(A)$ is a cyclic group and
thus a finite cyclic $p$-group. In this case, a generator is given
by the logarithmic derivative of a parameter on the affine line.
We discuss this criterion for forms of Russell type.

It is not unlikely that one may obtain more precise results for
the Picard group by developing a graded version of $p$-radical descent theory and then incorporate the gradings into all our arguments.
We leave this as an open question for future work.

\vskip8pt

We assemble some notions and results of graded commutative algebra
in an appendix.



~\\{\it Acknowledgement.} I thank Michael Temkin and Brian
Conrad for explaining to me some points in analytic geometry.

\section{Tamely ramified forms}

\subsection{$G$-groups}

A reference for the following is \cite[I.\S5]{SerreG}. Let $G$ be
a finite group. A {\it $G$-group} is a set $A$ with a $G$-action
together with a group structure which is invariant under $G$ (i.e.
$g(aa')=g(a)g(a')$). Let $A$ be a $G$-group. We denote by $A^G$ or
$H^0(G,A)$ the subgroup of $A$ consisting of elements $a$ with
$g(a)=a$ for all $g\in G$. A {\it (1-)cocycle of $G$ in $A$} is a map
$g\mapsto a_g$ from $G$ to $A$ such that
$$a_{gh}=a_g \hskip3pt g(a_h)$$
for all $g,h\in G$. Two cocycles $a_g$ and $a'_g$ are {\it
cohomologous} if there is $b\in A$ such that $ba_g=a'_{g}g(b)$ for
all $g\in G$. This is an equivalence relation on the set of
cocycles and the set of equivalence classes is denoted by
$H^1(G,A)$. The set $H^1(G,A)$ is a {\it pointed set}, i.e. a set
with a distinguished element, namely the class of the trivial
cocycle $a_g=1$ for all $g\in G$. A {\it morphism of pointed sets}
is a map preserving distinguished elements. There is an obvious
notion of kernel and exact sequence for pointed sets and their
morphisms.

The sets $H^0(G,A)$ and $H^1(G,A)$ are functorial in $A$ and
coincide with the usual cohomology groups of dimension $0$ and $1$
in case $A$ is abelian.

\vskip8pt

Given a normal and $G$-invariant subgroup $B$ of $A$ the quotient $C=A/B$ is a
$G$-group and there is an exact sequence of pointed sets
\begin{numequation}\label{equ-longexact} 1\longrightarrow B^G\longrightarrow A^G\longrightarrow
C^G\longrightarrow H^1(G,B)\longrightarrow H^1(G,A)\longrightarrow
H^1(G,C).\end{numequation}

Given a normal subgroup $H\subseteq G$ there is an exact sequence
of pointed sets
\begin{numequation}\label{equ-longexactII} 1\longrightarrow H^1(G/H,A^H)\longrightarrow H^1(G,A)\longrightarrow
H^1(H,A)\end{numequation} and the map $H^1(G/H,A^H)\rightarrow
H^1(G,A)$ is injective.

\subsection{Twisted forms and Galois
cohomology}\label{subsection-twisted}

A reference for the following is \cite[II.\S9]{KnusOjanguren}.
Consider an extension $R\subseteq S$ of commutative rings and a
finite subgroup $G$ of automorphisms of the $R$-algebra $S$. We
assume that the extension is {\it Galois} with Galois group $G$.
This means that $S$ is a finitely generated projective $R$-module
whose endomorphism ring ${\rm End}_R(S)$ admits a basis as (left)
$S$-module consisting of $\sigma\in G$. The table of
multiplication is therefore
$(s\sigma)(t\tau)=s\sigma(t)\sigma\tau$ for all $s,t\in S$ and
$\sigma,\tau\in G$.

Let $B$ be an $R$-algebra and let $B_S:=S\otimes_R B$. The ring
$B_S$ has the obvious $G$-action induced by $g(s\otimes
b)=g(s)\otimes b$. Let $\aut$ be the group of automorphisms of the
$S$-algebra $B_S$. Then $\aut$ becomes a $G$-group via
$g.\alpha=g\alpha g^{-1}$, i.e.
$$(g.\alpha)(s\otimes b)=g(\alpha(g^{-1}(s)\otimes b))$$
where $\alpha\in\aut$. Let ${\bf \Aut}\hskip2pt B$ be the
automorphism functor of $B$.

\vskip8pt

Let now $(A)$ be an isomorphism class of $R$-algebras and $A$ a
representative of this class. We call $(A)$ a {\it twisted form of
$B$ with respect to $R\subseteq S$} if there is an isomorphism of
$S$-algebras
$$\beta: S\otimes_R A\car S\otimes_R B.$$ Let $H^1(S/R, {\bf \Aut}\hskip2pt B)$ be the set of such
forms (by abuse of language we will also call {\it forms}
the representatives $A$ of a form $(A)$). It is a pointed set,
the distinguished element being $(B)$. Given a form $(A,\beta)$
the map $g\mapsto\theta_g$ where
$$\theta_g=\beta g\beta^{-1}g^{-1}$$
is a cocycle of $G$ in $\aut$. If $(A',\beta)$ is a different
representative of $(A)$ with cocycle $\theta_g'$, one has with
$\alpha:=\beta'\beta^{-1}$ that
$$ \alpha\theta_g=
\beta'g\beta^{-1}g^{-1}=\beta'g\beta'^{-1}g^{-1}
g\beta'\beta^{-1}g^{-1}=\theta_g'g.\alpha$$ for $g\in G$, i.e.
$\theta_g$ and $\theta'_g$ are cohomologous. We obtain a bijection
of pointed sets \begin{numequation}\label{equ-isoclasses}H^1(S/R,
{\bf \Aut}\hskip2pt B)\car H^1(G,\aut)\end{numequation} according
to loc.cit., II. Thm. 9.1. The inverse map is explicitly given as
follows. Let $\theta_g$ be a cocycle of $G$ in $\aut$. Given $g\in
G$ we let $\bar{g}:=\theta_g g.$ Hence any $\bar{g}$ acts
$S$-semilinearly on $B_S$. Let
$$A:=\{x\in B_S: \bar{g}(x)=x {\rm ~for~all~}g\in G\}.$$
Then $A$ is a form of $B$ the isomorphism $\beta:A_S\car B_S$
being induced by the inclusion $A\subset B_S$. The class $(A)$ is
the preimage of the class of $\theta_g$.

\subsection{Graded forms}
Let $\Gamma=\bbR_{+}^\times$ (or any other commutative
multiplicative group). The reader is referred to the appendix for
all occuring notions from graded ring theory that we will use in
the following. All occuring graded rings will be $\Gamma$-graded
rings and so we will frequently omit the group $\Gamma$ from the
notation. Let $k$ be a graded field and denote by $\rho:
k^\times\rightarrow\Gamma$ its grading.

\vskip8pt

Let $k\subseteq\ell$ be an extension (finite or not) of graded
fields and let $B$ be a graded $k$-algebra. Then
$B_\ell=\ell\otimes_k B$ is a graded $\ell$-algebra with respect
to the tensor product grading

$$(B_\ell)_\gamma =\sum_{\delta\tau=\gamma} \ell_\delta \otimes
B_\tau$$ for $\gamma\in\Gamma$. Let ${\bf \Aut}^{\rm gr}\; B$ be
the automorphism functor of $B$. A graded $k$-algebra $A$ is a
{\it graded} twisted form of $B$ with respect to $k\subseteq\ell$
if there is an isomorphism (of degree $1$) of graded
$\ell$-algebras $A_\ell\car B_\ell$. As above we then have the
pointed set $H^1(\ell/k, {\bf \Aut}^{\rm gr} B)$ of graded forms of
$B$ relative to the extension $\ell/k$. Let $\autgrlo$ be the
group of automorphisms of the graded $\ell$-algebra $B_\ell$. If
$k\subseteq\ell$ is a finite Galois extension with group $G$, then
$\autgrlo$ becomes a $G$-group in the same manner as above. In
this case, we have the cohomology set $H^1(G,\autgrlo)$ and the
isomorphism of pointed sets
\begin{numequation}\label{equ-grisoclasses}H^1(\ell/k, {\bf
\Aut}^{\rm gr} B)\car H^1(G,\autgrlo)\end{numequation} constructed
as in the ungraded case.

\subsection{Graded reductions}\label{subsec-gradedreductions}

We let $\Gamma=\bbR^\times_+$ have its usual total ordering. Let
$k$ be a field endowed with a nontrivial nonarchimedean absolute
value $|.|: k\rightarrow\mathbb{R}_+$. Let $A$ be a $k$-algebra
with a submultiplicative nonarchimedean seminorm $|.|:
A\rightarrow\mathbb{R}_+$ extending the absolute value on $k$ (and
therefore denoted by the same symbol). We let $\tgA$ be the graded
ring equal to
$$\tgA:=\oplus_{\gamma\in\Gamma}\hskip3pt\Aleq/\Aeq$$
where $\Aleq$ and $\Aeq$ consists of the elements $a\in A$ with
$|a|\leq\gamma$ and $|a|<\gamma$ respectively. Following
\cite{Ducros} we call $\tgA$ the {\it graded reduction} of $A$ in
the sense of M. Temkin \cite{TemkinI},\cite{TemkinII}. The
homogeneous part $\tA_1$ is a subring, the {\it residue ring} of
$A$ (in the sense of \cite{BGR}).

\vskip8pt

If $r\in\Gamma$ and $a\in A$ is an element with $|a|\leq r$ we
denote by $\ta_r$ the corresponding element in $A_{\leq
r}/A_{<r}$. If $|a|=r$ we simply write $\ta$ instead of $\ta_r$
and call $\ta$ the {\it principal symbol} of $a$. If $|a|=0$ we
put $\ta=0$.

\vskip8pt

In the case $A=k$ the graded reduction $\tgk$ is a graded field:
indeed, any homogeneous nonzero element is a principal symbol
$\ta$ and the principal symbol of $a^{-1}$ provides the inverse.
It is called the {\it graded residue field} of $k$. The
homogeneous part $\tk_1$ is a field, the {\it residue field} of
$k$ (in the sense of \cite{BGR}).

\vskip8pt

Suppose $k\subseteq\ell$ is a finite extension between
nonarchimedean fields where the absolute value on $\ell$ restricts
to the one on $k$. Let $$e:=(|\ell^\times|:|k^\times|)\hskip10pt
{\rm and}\hskip10pt f:=[\tell_1:\tk_1].$$ Then
$\tgk\subseteq\tgell$ is a finite extension of graded fields with

$$ef=[\tgell:\tgk]\leq [\ell:k]$$
as follows from \cite[Prop. 2.10]{Ducros}.

\vskip8pt Remark: According to \cite[Prop. 3.6.2/4]{BGR} the
inequality $[\tgell:\tgk]\leq [\ell:k]$ being an equality is
equivalent to the extension $\ell$ being {\it $k$-cartesian} (in
the sense of loc.cit., Def. 2.4.1/1 ). In loc.cit., the ground
field $k$ is defined to be {\it stable} if this property holds for
any finite extension of $k$. Any nonarchimedean field complete
with respect to a discrete valuation is stable (loc.cit., Prop.
3.6.2/1). Any nonarchimedean field which is spherically complete
(or equivalently, maximally complete \cite{Kaplansky}) is
stable \cite[Prop. 3.6.2/12]{BGR}. Finally, any nonarchimedean
complete field which is algebraically closed is stable (loc.cit.,
Prop. 3.6.2/12).

\vskip8pt

Denote by $p$ the {\it characteristic exponent} of the field $k_1$, i.e.
$p={\rm char}\hskip2pt k_1$ in case ${\rm char}\hskip2pt k_1>0$
and $p=1$ else. The extension $k\subseteq\ell$ is called {\it
tamely ramified} if the residue field extension
$\tk_1\subseteq\tell_1$ is separable and $e$ is prime to $p$. This
is equivalent to the extension of graded fields
$\tgk\subseteq\tgell$ being separable \cite[Prop. 2.10]{Ducros}.
We call $k\subseteq\ell$ {\it wildly ramified} if the residue
field extension $\tk_1\subseteq\tell_1$ is purely inseparable and
$e$ is a $p$-power.
\begin{lemma}\label{lem-purelyinsep}
The extension $k\subseteq\ell$ is wildly ramified if and only if
$\tgk\subseteq\tgell$ is purely inseparable.
\end{lemma}
\begin{proof}
If $\tk_1\subseteq\tell_1$ is purely inseparable, so is the
extension $\tgk\subseteq \tell_1\cdot \tgk$. Take
$\tilde{x}\in\tgell$ for some $x\in\ell^\times$. If $e=p^n$, say,
then $|x^{p^n}|\in |k^\times|$. This means there is $z\in\tell_1$
and a homogeneous nonzero $w\in\tgk$ such that
$\tilde{x}^{p^n}=zw$ and therefore $\tilde{x}$ is purely
inseparable over $\tell_1\cdot\tgk$. Conversely, suppose that
$\tgk\subseteq\tgell$ is purely inseparable. The homogeneous
minimal polynomial $f$ over $\tgk$ of any nonzero homogeneous
element $x$ of $\tgell$ is purely inseparable. If $x\in\tell_1$,
then $f$ has coefficients in $\tk_1$. Hence $\tell_1$ is purely
inseparable over $\tk_1$. Since $[\tgell:\tgk]$ is a $p$-power, so
is $e$.
\end{proof}
Remark: Let $\ell/k$ be an arbitrary $k$-cartesian finite
extension. If $f=1$ (e.g. if $\ell/k$ is wildly ramified with
$\tk_1$ being perfect), then $[\ell:k]=e$ and so $\ell/k$ is {\it
totally ramified}. 

\subsection{Tame ramification and Galois cohomology}
Let $k$ be a field complete with respect to a nontrivial
nonarchimedean absolute value $|.|$. Let $p$ be the characteristic
exponent of the residue field $\tk_1$ of $k$.

\vskip8pt

We will work relatively to a finite tamely ramified field extension
$k\subseteq\ell$ which is Galois with group $G$. Let $n=[\ell:k]$.
The finite extensions between the fields $\tk_1\subseteq\tell_1$
and the graded fields $\tgk\subseteq\tgell$ are then normal
\cite[Prop. 2.11]{Ducros} and therefore Galois and have degrees
$f$ and $n$ respectively. Since $G$ acts by isometries on $\ell$
we have two natural homomorphisms
\begin{numequation}\label{equ-galoisident}G=\Gal(\ell/k)\car \Gal(\tgell/\tgk)\longrightarrow
G(\tell_1/\tk_1)\end{numequation} for the corresponding Galois
groups. Both maps are surjective and the first map is even an
isomorphism (loc.cit., Prop. 2.11). Let $I=I(\ell/k)$ be the {\it
inertia subgroup}, i.e. the kernel of the composite homomorphism.
Then $\Gal(\ell/k)/I\car \Gal(\tell_1/\tk_1)$ and $n=ef$ implies $\#
I=e$. Hence $I$ has no $p$-torsion which implies the familiar
isomorphism
\begin{numequation}\label{equ-inertia}I\car {\rm \Hom}(|\ell^\times|/|k^\times|,
\tell_1^\times),\hskip3pt g\mapsto \psi_g\end{numequation} where
$\psi_g(\gamma\hskip3pt {\rm mod}\hskip3pt |k^\times|)$ equals the
reduction of $\frac{g(x)}{x}$ for some element $x\in\ell^\times$
of absolute value $\gamma$ (loc.cit., Prop. 2.14). In particular,
$I$ is abelian.

\vskip8pt

In the following we will collect some results on cohomology groups
associated with the ring of integers $\ell^\circ$ in $\ell$ and
its residue field $\tell_1$. These results are certainly
well-known but we give complete proofs for lack of suitable
reference.

\begin{lemma}\label{lem-normalbasis}
One has $H^n(G,\tell_1)=0$ for all $n\geq 1$.
\end{lemma}
\begin{proof}
We have $H^n(I,\tell_1)=0$ for all $n\geq 1$ since the order $\#I$
is invertible in $\tell_1$ and therefore the group ring
$\tell_1[I]$ is semisimple. Now $I$ acts trivially on $\tell_1$ and
the normal basis theorem \cite[Prop. X.\S1.1]{SerreL} implies
$H^n(G/I, \tell_1)=0.$ The assertion follows therefore from the
exact inflation-restriction sequence
$$ 0\longrightarrow H^n(G/I, (\tell_1)^{I})\longrightarrow
H^n(G,\tell_1)\longrightarrow H^n(I,\tell_1)$$ (loc.cit., Prop.
VII.\S6.5).
\end{proof}
Let $\ell^\circ\subset\ell$ be the subring of elements $x$ with
$|x|\leq 1$. Let $G$ act trivially on the value group
$|\ell^\times|$. We have the $G$-equivariant short exact sequence
$$
1\longrightarrow(\ell^\circ)^\times\longrightarrow\ell^\times\stackrel{|.|}{\longrightarrow}
|\ell^\times|\longrightarrow 1.$$
\begin{prop}\label{prop-ram}
The above sequence induces an isomorphism $$\vg\car
H^1(G,(\ell^{\circ})^\times)$$ given by
$\gamma=|x|\mapsto\psi(\gamma)$ where
$\psi(\gamma)_g=\frac{gx}{x}$ for all $g\in G$.
\end{prop}
\begin{proof}
The long exact cohomology sequence gives
$$ k^\times=H^0(G,\ell^\times)\stackrel{|.|}{\longrightarrow}
H^0(G,|\ell^\times|)\stackrel{\delta}{\longrightarrow}
H^1(G,(\ellc)^\times)\longrightarrow H^1(G,\ell^\times)=1$$ where
the right hand identity is {\it Hilbert theorem 90} (e.g.
\cite[Prop. X.\S1.2]{SerreL}). The boundary map $\delta$ gives the
required map.
\end{proof}
The group $G/I$ acts on $I$ via conjugation since $I$ is abelian
and it acts on $\tell_1^\times$ via the natural action. We
therefore have an induced action of $G/I$ on the set of group
homomorphisms ${\rm \Hom}(I,\tell_1^\times)$ via
$(g.f)(h)=gf(g^{-1}h)$ for all $h\in I$.

\begin{lemma}Let $\ell^{\circ\circ}$ be the kernel of the reduction map
$\ell^\circ\rightarrow \tell_1$. There is an isomorphism
$$\vg\car {\rm \Hom}(I,\tell_1^\times)^{G/I}={\rm \Hom}(I,\tell_1^\times)$$ mapping
$\gamma=|x|$ to $g\mapsto (\psi(\gamma)_g{\rm ~mod~}\ellcc)$.
\end{lemma}
\begin{proof}
Quite generally, if $M,A$ are abelian $G$-groups where $G$ acts
trivially on $M$ and the sets of group homomorphisms $M^*:={\rm
\Hom}(M,A)$ and $M^{**}:={\rm \Hom}(M^*,A)$ have their induced
$G$-actions, then the natural bidual homomorphism $M\rightarrow
M^{**}$ is equivariant. This is elementary. Now let $M:=\vg$ and
$A:=\tell_1^\times$ and recall the isomorphism $I\car(\vg)^*$ from
(\ref{equ-inertia}). Since $\#\vg=e=\# I$ the bidual map
$\vg\car(\vg)^{**}$ is bijective. By the above remark it is
equivariant (and so the action on both sides is trivial). But
$I\car(\vg)^*$ is also equivariant. Indeed, let $g_0\in G, g\in I$
and $\gamma=|x|$. Then
$$ \frac{(g_0.g)(x)}{x}=\frac{g_0gg_0^{-1}(x)}{x}=
g_0(\frac{gg_0^{-1}(x)}{g_0^{-1}(x)})$$ and since
$\gamma=|g_0^{-1}(x)|$ passing to the reduction mod ${\rm \ellcc}$
implies
$\psi_{g_0.g}(\gamma)=g_0(\psi_g(\gamma))=(g_0.\psi_g)(\gamma).$
The bidual map gives therefore an equivariant isomorphism
$$\vg\car {\rm \Hom}(I,\tell_1^\times)$$
(and so the action on both sides is trivial). The definition of
the map (\ref{equ-inertia}) shows that this isomorphism has the
required form.
\end{proof}

\begin{prop}\label{prop-unitsvsresiduefield}
The reduction map $\ell^\circ\rightarrow \tell_1$ induces an
isomorphism $$\vg=H^1(G,(\ell^\circ)^\times)\car
H^1(G,\tell_1^\times).$$
\end{prop}
\begin{proof}
The inflation-restriction exact sequence
$$ 1\longrightarrow H^1(G/I,\tell^\times_1)\longrightarrow
H^1(G,\tell^\times_1)\longrightarrow H^1(I,\tell^\times_1)^{G/I}$$
together with the {\it Hilbert theorem 90} and the preceding lemma
yields an injection
$$ H^1(G,\tell_1^\times)\hookrightarrow {\rm
\Hom}(I,\tell_1^\times)=\vg.$$ Precomposing this injection with the
map $\vg=H^1(G,(\ell^\circ)^\times)\longrightarrow
H^1(G,\tell_1^\times)$ gives the identity on $\vg$. This yields
the assertion.
\end{proof}

\subsection{Spectral norms and affinoid algebras}
We work in the framework of Berkovich analytic spaces
\cite{Berkovichbook}. However, we deal only with affinoid
algebras and our methods are entirely algebraic. We let $k$ be a
nonarchimedean field which is complete with respect to a
nontrivial absolute value $|.|$. Let $\Gamma=\bbR^\times_+.$ 

\vskip8pt

In the following, graded reductions of affinoid algebras $A$ are
always computed relatively to their spectral seminorm
$$ |f|=\sup_{x\in\sM(A)} |f|_x$$
for $f\in A$ and where $\sM(A)$ denotes the corresponding affinoid
space. If the ring $A$ is reduced, this seminorm is a norm and
induces the Banach topology on $A$ \cite[Prop. 6.2.1/4 and Thm.
6.2.4/1]{BGR}.

\vskip8pt

Suppose $k\subseteq\ell$ is an extension between nonarchimedean
fields where the absolute value on $\ell$ restricts to the one on
$k$.

\begin{lemma}\label{lem-keyreduced}
Let $A$ be a reduced $k$-affinoid algebra with its spectral norm.
If the ring $\tgell\otimes_{\tgk}\tgA$ is reduced, then so is the
ring $A_\ell:=\ell\hskip2pt\hat{\otimes}_{k}\hskip2pt A$. In this
case, the graded reduction of the $\ell$-affinoid algebra $A_\ell$
equals $\tgell\otimes_{\tgk} \tgA$.
\end{lemma}
\begin{proof}
Consider the tensor product norm on $A_\ell$. The corresponding
graded reduction equals $\tgell\otimes_{\tgk}\tgA$ because the
extension of graded fields $\tgk\subseteq\tgell$ is free. Since
this graded reduction is reduced, the tensor product norm must be
power-multiplicative (i.e. $||f^n||=||f||^n$) and therefore
$A_\ell$ must be reduced. Since the tensor product norm is
furthermore a complete norm on $A_\ell$, it must be equal to the
spectral norm on the $\ell$-affinoid algebra $A_\ell$ \cite[Prop.
6.2.3/3]{BGR}.
\end{proof}

Let $r\in \Gamma$ and denote by $B=k\{r^{-1}T\}$ the $k$-affinoid
algebra of the closed disc of radius $r$. We write $B=k\{T\}$ in
case $r=1$. The spectral seminorm on $B$ is a multiplicative norm
given by
$$ |\sum_n a_nT^n|=\sup_n |a_n|r^n.$$ The associated graded
reduction is given by $\tgB=\tgk[r^{-1}T]$ according to \cite[Prop.
3.1.i]{TemkinII}. 

\begin{lemma}\label{lem-gradediso}
Let $A$ be a reduced $k$-affinoid algebra such that
$\tgell\otimes_{\tgk} \tgA$ is reduced for any complete extension
field $\ell$ of $k$. Let $s\in\bbR_+$. If $f:
k\{s^{-1}T\}\rightarrow A$ is a homomorphism of $k$-algebras whose
graded reduction is an isomorphism, then $f$ is an isomorphism.
\end{lemma}
\begin{proof}
Since the graded reduction of $f$ is injective, $f$ is isometric
and therefore injective. Let $B':=k\{s^{-1}T\}$. Consider any
complete extension field $k\subseteq \ell$. If the induced map
$f_\ell:=\ell\hat{\otimes}_k f$ between
$B'_\ell:=\ell\hat{\otimes}_k B'$ and $A_\ell:=\ell\hat{\otimes}_k
A$ is surjective, then so is $f$. Indeed, $f$ is strict and has
therefore closed image. We thus have a strict exact sequence
$B'\rightarrow A\rightarrow {\rm coker} f\rightarrow 0$ of
$k$-Banach spaces. By results of L. Gruson \cite[Lem.
A5]{RemyThuillierWerner10} we obtain an embedding ${\rm coker}
f\hookrightarrow \ell\hskip3pt\hat{\otimes}_k\hskip3pt {\rm coker}
f ={\rm coker} f_\ell$, which proves the claim.

There is a complete extension field $k\subseteq\ell'$ such that
$B'_{\ell'}$ and $A_{\ell'}$ are both strictly $\ell'$-affinoid
and such that $|B'_{\ell'}|=|\ell'|$, cf. \cite[2.1]{Berkovichbook}.
We let $\ell$ be the completion of an algebraic closure of
$\ell'$. The graded reductions of the $\ell$-affinoid algebras
$B'_\ell$ and $A_\ell$ are given by
$\tilde{\ell}_\bullet\otimes_{\tgk}\tilde{B'}_\bullet$ and
$\tilde{\ell}_\bullet\otimes_{\tgk}\tgA$ respectively. For $B'$
this follows from \cite[Prop. 3.1(i)]{TemkinII}. For $A$ this
follows from our hypothesis and Lem. \ref{lem-keyreduced}. We see
that the graded reduction of $f_\ell$ is the base change to
$\tgell$ of $\tilde{f}_\bullet$ and therefore still an
isomorphism. In particular, the ordinary reduction
$\tilde{(f_\ell)}_1$ is an isomorphism. Since
$\tilde{(A_\ell)}_\bullet $ and $\tilde{(B'_\ell)}_\bullet$ are
reduced, $A_\ell$ and $B'_\ell$ are reduced (Lem.
\ref{lem-keyreduced}). According to \cite[Prop. 3.4.1/3]{BGR} the
field $\ell$ is algebraically closed. Hence, $\ell$ is a stable
field and $|\ell^\times|$ is divisible (loc.cit., Prop. 3.6.2/12).
We may now apply loc.cit., Cor. 6.4.2/2 to conclude from the
bijectivity of $\tilde{(f_\ell)}_1$ that $f_\ell$ is bijective.
\end{proof}

In the following we are interested in affinoid algebras $A$ that
allow an isomorphism
$$\beta: \ell\hskip2pt \hat{\otimes}_k A\car \ell\hskip2pt \hat{\otimes}_k B$$ of $\ell$-algebras (and
hence of $\ell$-affinoid algebras). We denote the pointed set of
isomorphism classes of such $k$-algebras $A$ by $H^1(\ell/k, {\bf
\Aut}\hskip2pt B)$. If $\ell/k$ is a finite Galois extension and if
$A$ is an {\it arbitrary} $k$-algebra with an isomorphism $\beta$
as above, then $A$ is $k$-affinoid \cite[Prop.
2.1.14(ii)]{Berkovichbook}. Hence our present notation is
consistent with subsection \ref{subsection-twisted}.

\vskip8pt

Let us now assume that $k\subseteq\ell$ is a finite Galois
extension. If $G:=\Gal(\ell/k)$ denotes the Galois group, the
isomorphism (\ref{equ-isoclasses}) implies $H^1(\ell/k, {\bf
\Aut}\hskip2pt B)\simeq H^1(G,\autl).$ We therefore will focus on
the Galois cohomology $H^1(G,\autl).$

\vskip8pt

Let $(A,\beta)$ be a form of $B$. Recall that a commutative ring
of dimension $1$ which is noetherian and integrally closed is
called a {\it Dedekind domain}.
\begin{lemma}\label{lem-Dedekind}
The ring $A$ is a Dedekind domain.
\end{lemma}
\begin{proof}
Any affinoid algebra is noetherian. Since $\beta$ exhibits
$B_\ell$ as a finite free $A$-module, $A$ is an integral domain of
dimension $1$. The semilinear $G$-action on $B_\ell$ extends
naturally to its fraction field ${\rm Frac}(B_\ell)$ with ${\rm
Frac}(B_\ell)^G={\rm Frac}(A)$. Suppose $f\in {\rm Frac}(A)$ is
integral over $A$. Since $B_\ell$ is integrally closed, we have
$f\in B_\ell\cap {\rm Frac}(B_\ell)^G=(B_\ell)^G=A$.
\end{proof}
The isomorphism $\beta$ restricts to an inclusion
$A\hookrightarrow B_\ell$ which is strict with respect to spectral
seminorms. In particular, $\tgA\hookrightarrow\tgellB$ and, hence,
$\tgA$ is an integral domain. Since $\tgellB$ is a graded
$\tgell$-algebra, the inclusion extends to a homomorphism
\begin{numequation}\label{equ-mainhomo}\tgell\otimes_{\tgk}\tgA\rightarrow\tgellB\end{numequation}
of graded $\tgell$-algebras. It will play a central role in the
following. In general, it is neither injective nor surjective (cf.
example after proof of \cite[Prop. 3.1]{TemkinII}).

\vskip8pt

The maximal tamely ramified subextension

$$ k\subseteq k^{\rm tr}\subseteq\ell$$

induces, according to (\ref{equ-longexactII}), a short exact
sequence

$$ 1\longrightarrow  H^1(\Gal(k^{\rm tr}/k), \Aut\hskip2pt B_{k^{\rm tr}}) \longrightarrow
H^1(G, \Aut \hskip2pt B_\ell)\longrightarrow H^1(\Gal(\ell/k^{\rm
tr}), \Aut\hskip2pt B_\ell)$$ where the map $H^1(\Gal(k^{\rm
tr}/k), \Aut\hskip2pt B_{k^{\rm tr}}) \longrightarrow H^1(G,
\Aut\hskip2pt B_\ell)$ is injective. In the following we will
examine the outer two terms in this sequence.

\vskip8pt

Let
$$\sqrt{|k^\times|}:=\{\alpha\in\bbR_+: \alpha^n\in |k^\times| {\rm
~for~some~}n\geq 1\}.$$ If $r\in \sqrt{|k^\times|}$, we may
enlarge $\ell$ so that $r=|\epsilon|\in |\ell^\times|$ for some
$\epsilon\in\ell^\times$ and compose $\beta$ with the isomorphism
$B_\ell=\ell\{r^{-1}T\}\car\ell\{T\}$ induced by $T\mapsto
\epsilon T$. We will therefore restrict ourselves in the following
to the two cases $$r=1 \hskip10pt {\rm or~}\hskip10pt r\notin
\sqrt{|k^\times|}.$$

\subsection{Tamely ramified forms}

We keep all assumptions, but we suppose additionally that the
finite Galois extension $k\subseteq\ell$ is {\it tamely ramified}.
In this case we can determine all corresponding forms of $B$ based
on the following key proposition.

\begin{prop}\label{prop-gradedform} The graded homomorphism (\ref{equ-mainhomo}) is
bijective.
\end{prop}
\begin{proof}
Let $F:={\rm Frac}_\Gamma(\tgA)$ be the graded fraction field of
$\tgA$, cf. appendix. Since $k\subseteq\ell$ is tamely ramified,
the finite extension $\tgk\subseteq\tgell$ is separable. The
finite graded $F$-algebra $F\otimes_{\tgk}\tgell$ is therefore
\'etale in the sense of \cite[1.14.4]{Ducros} and therefore
reduced. Thus, the subring
$\tgA\otimes_{\tgk}\tgell\hookrightarrow F\otimes_{\tgk}\tgell$ is
reduced. Now Lem. \ref{lem-keyreduced} implies that
(\ref{equ-mainhomo}) is isomorphic to the graded reduction of
$\beta$. In particular, it is bijective.
\end{proof}
The proposition implies a map of pointed sets
$$H^1(\ell/k, {\bf \Aut}\hskip2pt B)\longrightarrow H^1(\tgell/\tgk, {\bf \Aut}^{\rm gr}\; \tgB),\hskip8pt(A)\mapsto (\tgA).$$
By (\ref{equ-galoisident}) we may identify $G=G(\ell/k)\simeq
G(\tgell/\tgk).$ The obvious group homomorphism
$\autl\rightarrow\autgrl$ is then $G$-equivariant and induces a
map of pointed sets
$$H^1(G,\autl)\rightarrow H^1(G,\autgrl).$$ The following lemma
follows by unwinding the definitions of the maps involved.
\begin{lemma}
The diagram
\[\xymatrix{
 H^1(\ell/k, {\bf \Aut}\hskip2pt B) \ar[r] \ar[d]^{\simeq} & H^1(\tgell/\tgk, {\bf \Aut}^{\rm gr}\; \tgB)  \ar[d]^{\simeq} \\
  H^1(G,\autl) \ar[r] & H^1(G,\autgrl) }
\]

is commutative. Here, the vertical arrows come from
(\ref{equ-isoclasses}) and (\ref{equ-grisoclasses}).
\end{lemma}
An automorphism $\alpha$ of the graded $\tgell$-algebra
$\tgellB=\tgell[r^{-1}T]$ is, in particular, an automorphism of
the underlying algebra and maps $T$ to $aT+b$ with $a,b\in\tgell$.
Since $\alpha$ preserves the grading, one has $a\in\tell_1^\times$
and a $G$-equivariant map $\autgrl\mapsto \tell_1^\times,
\alpha\mapsto a$.
\begin{prop}
The map induces an
isomorphism $H^1(G,\autgrl)\car H^1(G,\tell_1^\times)$.
\end{prop}
\begin{proof}
Suppose $r\notin\sqrt{|k^\times|}$. Then $b=0$ and already the map
$\alpha\mapsto a$ is an isomorphism. Suppose on the contrary
$r=1$. Then $\alpha\mapsto a$ induces an exact $G$-equivariant
sequence
$$
0\longrightarrow\tell_1\longrightarrow\autgrl\longrightarrow\tell_1^\times\longrightarrow
1.$$ It is equivariantly split. Indeed, given $a\in\tell_1^\times$
the map induced by $T\mapsto aT$ lies in $\autgrl$ and this gives
the equivariant splitting $\tell_1^\times\rightarrow\autgrl$.
By (\ref{equ-longexact}) the sequence
$$H^1(G,\tell_1)\longrightarrow H^1(G,\autgrl)\longrightarrow
H^1(G,\tell_1^\times)$$ is exact. We have a section for the right
hand arrow which is therefore surjective. It remains to see
$H^1(G,\tell_1)=0$. This follows from Lem. \ref{lem-normalbasis}.
\end{proof}
If $a\in\ell$ with $|a|=1$, then $T\mapsto aT$ gives an element of
$\autl$. This gives a map
\begin{numequation}\label{equ-section}|\ell^\times|/|k^\times|=H^1(G,(\ellc)^\times)\longrightarrow
H^1(G,\autl).\end{numequation} The first identity here is due to
Prop. \ref{prop-ram}. Conversely, any $\alpha\in\autl$ preserves
the spectral norm whence
$$\alpha(T)=a_0+a_1T+a_2T^2+\cdot\cdot\cdot$$
with $a_i\in\ell$ and
$$ |a_0|\leq r,\hskip8pt |a_1|=1 \hskip8pt{\rm and}\hskip8pt |a_i|r^{i}<r {\rm
~for~all~} 2\leq i.$$ Mapping $\alpha$ to $a_1$ induces a map
$$ H^1(G,\autl)\longrightarrow H^1(G,(\ellc)^\times)$$
with a section given by (\ref{equ-section}). Let us make the forms
of $B$ corresponding to the injection
$$ \vg=H^1(G,(\ell^\circ)^\times)\hookrightarrow H^1(G,\autl)$$
explicit. Let $x$ be an element of $\ell^\times$. For $g\in G$ let
$\theta_g$ be the automorphism of the $\ell$-algebra $B_\ell$
whose value on $T$ is given by $\frac{g(x)}{x}T$. The form of $B$
corresponding to the class of $\gamma=|x|$ in $\vg$ equals, up to
isomorphism, the subring $A\subseteq B_\ell$ of invariants under
the semilinear $G$-action given by $\bar{g}=\theta_gg$, cf.
subsection \ref{subsection-twisted} and proof of Prop.
\ref{prop-ram}. For an element $f=\sum_n a_nT^n\in B$ we compute
$$\bar{g}(f)=\sum_n g(a_n)(\theta_g(T))^n=\sum_n g(a_n)\frac{g(x^n)}{x^n}T^n$$ and thus
$$ \bar{g}(f)=f \hskip3pt{\rm for~all~}g\in G\Longleftrightarrow
g(a_n\hskip2pt x^n)=a_n\hskip1pt x^n
{\rm~for~all~}n,g\Longleftrightarrow a_n\hskip1pt x^n\in k
{\rm~for~all~}n.$$ The subring $A$ is therefore given by the ring
of all series
$$ \sum_n b_n(x^{-1}T)^n \hskip8pt {\rm where~}b_n\in k\hskip8pt
{\rm and~} |b_n|\hskip1pt s^n\rightarrow 0 \hskip8pt {\rm
for~}n\rightarrow\infty$$ where $s=\gamma^{-1}r$. Mapping
$x^{-1}T$ to $T$ induces an isomorphism $$A\car k\{s^{-1}T\}.$$
This shows that the forms of $B$ that correspond to the finite
subset $H^1(G,(\ellc)^{\times})$ are given by the $k$-affinoid
discs of radii $\gamma_ir$ where $\gamma_i\in |\ell^\times|$ runs
through a system of representatives for the classes in $\vg$.
\begin{prop}
One has a commutative diagram of bijections of pointed sets
\[\xymatrix{
  H^1(G,\autl) \ar[r]^>>>>>{\simeq} \ar[d]^{\simeq} & H^1(G,\autgrl)  \ar[d]^{\simeq} \\
H^1(G,(\ellc)^\times)\ar[r]^{\simeq} & H^1(G,\tell_1^\times)  }
\]
where the lower horizontal arrow comes from Prop.
\ref{prop-unitsvsresiduefield}.
\end{prop}
\begin{proof}
The commutativity of the diagram follows by inspection. Let
$(A,\beta)$ be a form of $B$ and let $c\in H^1(G,\autl)$ be its
cohomology class. Let $\tilde{c}$ be its image in
$H^1(G,\autgrl)$. Since the lower horizontal arrow and the right
vertical arrow in the diagram are isomorphisms, there is
$x\in\ell^\times$ such that $\tilde{c}=[\tilde{\theta}]$ where
$$\tilde{\theta}_g(T)=(\frac{gx}{x}{\rm~mod~}\ellcc)T\in \tgellB$$
for all $g\in G$. If $\tilde{x}\in\tgell$ denotes the principal
symbol of $x$, we have
$(\frac{gx}{x}{\rm~mod~}\ellcc)=\frac{g\tilde{x}}{\tilde{x}}.$ Let
$\gamma=|x|$. Now $\tgA$ equals, up to a graded isomorphism, the
invariant $\tgk$-subalgebra of $\tgellB$ under the semilinear
$G$-action defined by $\bar{g}:=\tilde{\theta}_g g.$ A computation
as above shows that this $\tgk$-subalgebra is given by the algebra
of all polynomials
$$ \sum_n b_n(\tilde{x}^{-1}T)^n\hskip8pt {\rm where~}b_n\in\tgk.$$
It has the induced grading from $\tgellB$, i.e. a polynomial as
above is homogeneous of degree $s$ if and only if
$b_n\tilde{x}^{-n}$ is homogeneous of degree $sr^{-n}$, i.e. if
and only if $b_n$ is homogeneous of degree $s\gamma^nr^{-n}$.
Mapping $\tilde{x}^{-1}T$ to $T$ induces a graded isomorphism from
this $\tgk$-subalgebra to $\tgk[s^{-1}T]$ where $s:=\gamma^{-1}r$.
We thus have an isomorphism of graded $\tgk$-algebras
\begin{numequation}\label{equ-red}\tgA\car\tgk[s^{-1}T].\end{numequation}
Let $\tilde{f}$ be the preimage of $T$ under this isomorphism and
$f\in A$ a representative. Mapping $T$ to $f$ induces a
homomorphism $\psi: k\{s^{-1}T\}\rightarrow A$ of $k$-affinoid
algebras whose graded reduction is an isomorphism. If $\ell'$ is
any complete extension field of $k$, then
$\tilde{\ell'}_\bullet\otimes_{\tgk}\tgA$ is reduced according to
(\ref{equ-red}). Since the ring $A$ is reduced, Lem.
\ref{lem-gradediso} implies that $\psi$ is an isomorphism. This
shows that $c$ lies in the image of
$H^1(G,(\ellc)^\times)\hookrightarrow H^1(G,\autl)$. The left
vertical arrow of our diagram is therefore an isomorphism and
hence, so is any arrow in the diagram.
\end{proof}

We summarize this discussion in the following statement.
\begin{thm}
The forms of $B$ with respect to the tamely ramified extension
$k\subseteq\ell$ are given by the finitely many $k$-affinoid discs
of radii $\gamma_ir$ where $\gamma_i\in |\ell^\times|$ runs
through a system of representatives for the classes in $\vg$.
\end{thm}

\section{Forms of Russell type and wild ramification}

\subsection{The additive group}\label{subsec-addgroup}

Let $k$ be a field of characteristic $p>0$ and let $\mathbb{G}_a$
be the additive group over $k$. The endomorphism ring of
$\mathbb{G}_a$ equals the noncommutative polynomial ring $k[F]$
with relations $Fa=a^pF$ for $a\in k$ where $F$ corresponds to the
Frobenius endomorphism \cite[Prop.
II.\S3.4.4]{DemazureGabriel}. We let $k[F]^*$ be the
multiplicative monoid consisting of separable endomorphisms, i.e.
the set of all polynomials $\sum_i a_iF^{i}$ with $a_0\neq 0$. For
any $n$ the left ideal $k[F]F^n$ of $k[F]$ is a two-sided ideal
and we let $U_n$ be the image of $k[F]^*$ under the projection
$k[F]\rightarrow k[F]/k[F]F^n$. Then $U_n$ is a multiplicative
group.

\vskip8pt

If $T$ denotes a parameter for $\mathbb{G}_a$, then $F$ induces
the ring homomorphism given by $T\mapsto T^p$ which we denote by
$F$ as well. Mapping
$$\tau=\sum_i a_iF^{i}\mapsto \tau (T)$$ yields a bijection of
$k[F]$ with the set of $p$-polynomials $
a_0T+a_1T^p+\cdot\cdot\cdot +a_mT^{p^m}$ in the (commutative)
polynomial ring $k[T]$. The separable endomorphisms correspond to
those $p$-polynomials with $a_0\neq 0$.

\vskip8pt

Any endomorphism $\varphi$ of $k$ extends to an endomorphism of
$k[F]$ via $$\tau=\sum_i a_iF^{i}\mapsto \varphi(\tau)=\sum_i
\tau(a_i)F^{i}.$$ If $\varphi$ equals the $n$-th Frobenius
$a\mapsto a^{p^n}$ on $k$ we write $\tau^{(n)}:=\varphi(\tau)$. In
this case $k[F]^{(n)}$ denotes the image of $k[F]$ under
$\tau\mapsto\tau^{(n)}$. This image is $k[F]$ if and only if $k$
is perfect.

\vskip8pt

Let $A$ be a Hopf algebra over $k$ with comultiplication $c$ and
let $T=0$ be the unit element of $\mathbb{G}_a$. The set of group
homomorphisms ${\rm \Hom}(\Spec(A), \mathbb{G}_a)$ is in bijection
via $\psi\mapsto\psi(T)$ with the set of elements $x\in A$ that
satisfy $c(x)=1\otimes x+ x\otimes 1$.

\vskip8pt

\subsection{Graded Frobenius}

Let $\Gamma=\mathbb{R}^\times_+$. 
All graded rings are $\Gamma$-graded rings. \vskip8pt

Let $A$ be a graded ring such that $A_1$ has characteristic $p>0$. 
Let $\rho=\rho_A$ be its grading function.  
Let $n\in\bbZ$. We define a new grading $\rho^{\tn}$ on $A$ by
$$\rho^{\tn}(x):=\rho(x)^{1/p^n}$$ for any nonzero homogeneous
$x\in A$. We denote by $A^{\tn}$ the ring $A$ equipped with the
new grading $\rho^{\tn}$, i.e. $A^{\tn}=\oplus_\gamma
A^{\tn}_\gamma$ where $$A^{\tn}_\gamma=A_{\gamma^{p^n}}$$ for all
$\gamma\in\Gamma$. Obviously $A^{\tn}_1=A_1$.

\vskip8pt

Let $k$ be a graded field such that $k_1$ has characteristic
$p>0$. Then $k^{\tn}$ is a graded field. If $A$ is a graded
$k$-algebra, then $A^{\tn}$ is a graded $k^{\tn}$-algebra. The
Frobenius map
$$\phi^n: k\rightarrow k^{\tn}, x\mapsto x^{p^n}$$ is a graded
homomorphism (of degree $1$) and we may form the graded tensor
product
$$\theta^{\tn}A=k^{\tn}\otimes_{\phi^n,k}A.$$ It is a graded
$k^{\tn}$-algebra via multiplication on the left hand factor and
comes with the Frobenius homomorphism
$$ F^{\tn}_A:
\theta^{\tn}A\longrightarrow A^{\tn},~x\otimes a\mapsto
xa^{p^n}.$$ The map $F^{\tn}_A$ is a graded homomorphism (of
degree $1$) of graded $k^{\tn}$-algebras.

\vskip8pt

We briefly discuss the relation between the functor $\theta^{\tn}$
and base change. Let $\bar{k}$ be a graded algebraic closure of
$k$, cf. appendix. Suppose that $k\subseteq\ell\subseteq\bar{k}$
is a graded extension field which is purely inseparable over $k$.
If $\ell^{p^n}\subseteq k$ for some $n$, then associativity of the
graded tensor product gives
\begin{numequation}\label{equ-FrobI}\theta^{\tn}A\simeq k^{\tn}\otimes_{\bar{\phi}^n,\ell}
A_\ell
\end{numequation} with the graded map
$\bar{\phi}^n:\ell\rightarrow k^{\tn}, x\mapsto x^{p^n}$. Let
$$k^{p^{-n}}:=\{x\in \bar{k}: x^{p^n}\in k\},$$ a graded
subfield of $\bar{k}$ and a purely inseparable extension of $k$. 
If $k^{p^{-n}}\subseteq\ell$, we have in turn
\begin{numequation}\label{equ-FrobII}\ell\otimes_{\psi^n,k^{\tn}}\theta^{\tn}A=\ell\otimes_{\psi^n,k^{\tn}}
(k^{\tn}\otimes_{\bar{\phi}^n,\ell} A_\ell)\simeq
A_\ell\end{numequation} where $\psi^n: k^{\tn}\rightarrow\ell,
x\mapsto x^{1/p^n}$.

\subsection{Graded forms of Russell
type}\label{subsection-gradedformsofRusselltype}

We keep the assumptions. Let $A$ be a graded $k$-algebra. We call
$A$ a {\it graded Hopf algebra} if $A$ is a Hopf algebra relative
to the ring $k$ in the usual sense \cite{Waterhouse} and if
the relevant maps -comultiplication, inverse, augmentation- are
graded homomorphisms (of degree $1$). In this case $\theta^{\tn}A$
and $A^{\tn}$ are graded Hopf algebras for all $n\in\bbZ$. One
verifies that $F^{\tn}_A$ respects these structures.

\vskip8pt

Example: Let $r\in\Gamma$ and consider $B=k[r^{-1}T]$. The Hopf
algebra structure on $B$ coming from $\mathbb{G}_a$ with origin
$T=0$ is given by
$$c(T)=1\otimes T +T \otimes 1,\hskip5pt i(T)=-T \hskip5pt{\rm
and}\hskip5pt \epsilon(T)=0$$ and is therefore graded. 
Conversely, any structure of graded Hopf algebra on $B$ arises
this way. Indeed, the discussion before \cite[Lemma 1.2]{Russell}
carries over to the graded setting and hence, if $t$ is a
homogeneous generator of the kernel of the augmentation, then
necessarily $c(t)=1\otimes t + t\otimes 1$ and $i(t)=-t$. The
canonical isomorphism
$$\theta^{\tn}(k[r^{-1}T]) \car k^{\tn}[r^{-1}T], \hskip8pt
x\otimes a\mapsto xa$$ is graded and transforms the Frobenius
$F^{\tn}_{B}$ into the graded homomorphism
\begin{numequation}\label{equ-Frob} k^{\tn}[r^{-1}T]\longrightarrow
(k[r^{-1}T])^{\tn}=k^{\tn}[r^{-1/p^n}T], \hskip8pt T\mapsto
T^{p^n}.\end{numequation}


\vskip8pt

Let $r\in\Gamma$. A $p$-polynomial
$f(T)=a_0T+a_1T^p+\cdot\cdot\cdot +a_mT^{p^m}\in k[r^{-1}T]$ is
homogeneous of degree $\gamma\in\Gamma$ if and only if $a_i\in k$
is homogeneous of degree $\gamma r^{-p^{i}}$.

\vskip8pt

Let $n\geq 1$ and let $\ell$ be the graded field
$$\ell:=k^{p^{-n}}=\{ a\in\bar{k}: a^{p^n}\in k\}.$$ It is a purely inseparable (possibly infinite) extension of $k$. Let $r\in \Gamma$ and let $A$ be a
graded form of $B:= k[r^{-1}T]$ with respect to
$\ell/k$ that is also a graded Hopf algebra. 
We start with the following general construction. Let $s\in\Gamma$
and choose
$$f(T_1)=a_0T_1+a_1T^p+\cdot\cdot\cdot +a_mT^{p^m}\in
k[r^{-p^n}T_1],$$ a homogeneous $p$-polynomial of degree $s^{p^n}$
with $a_0\neq 0$. The graded quotient
$$A:= k[r^{-p^n}T_1,s^{-1}T_2]/(T_2^{p^n}-f(T_1))$$ inherits a
Hopf structure from $k[r^{-p^n}T_1,s^{-1}T_2]$ coming from the $2$-dimensional additive group of with unit element $T_1=T_2=0$. Since $a_0\neq
0$ the tensor product $\ell\otimes_k A$ is reduced. 

\vskip8pt

For any polynomial $h\in k[T]$ we let $h^{\phi^n}$ be the
polynomial that comes from $h$ by applying $\phi^n$ to its
coefficients. 
Now let $x$ respectively $y$ be the residue class of $T_1$ respectively $T_2$ in
the quotient $A$. We then have
$$ 1\otimes a_0x= 1\otimes y^{p^n}-1\otimes
a_mx^{p^m}-...-1\otimes a_1x^p= (1\otimes
y^{p^{n-1}}-(a_m^{p^{n-1}}\otimes
x^{p^{m-1}}+...+a_1^{p^{n-1}}\otimes x))^p=: t_1^p$$ in
$\theta^{\tn}A$ with a homogeneous element $t_1$ of degree
$s^{p^{n-1}}$. Let $h:=a_0^{-1}f$ and $b_i:=a_0^{-1}a_i$. 
Then
$$ 1\otimes y^{p^n}=1\otimes f(x)=f^{\phi^n}(1\otimes
x)=h^{\phi^n}(1\otimes
a_0x)=h^{\phi^n}(t_1^p)=h^{\phi^{n-1}}(t_1)^p$$ in
$\theta^{\tn}A$. Since $\theta^{\tn}A$ is reduced, this implies
\begin{numequation}\label{equ-algo}1\otimes y^{p^{n-1}}=h^{\phi^{n-1}}(t_1)\end{numequation} in $\theta^{\tn}A$
viewing the coefficients of $h^{\phi^{n-1}}$ as elements of
$k^{\tn}$. Writing this out yields
$$ t_1= 1\otimes y^{p^{n-1}}-
b^{p^{n-1}}_mt_1^{p^m}-...-b^{p^{n-1}}_1t_1^p= (1\otimes
y^{p^{n-2}}-(b_m^{p^{n-2}} t_1^{p^{m-1}}+...+b_1^{p^{n-2}}
t_1))^p=: t_2^p$$ in $\theta^{\tn}A$ with a homogeneous element
$t_2$ of degree $s^{p^{n-2}}$. Now $$1\otimes
y^{p^{n-1}}=h^{\phi^{n-1}}(t_1)=h^{\phi^{n-1}}(t_2^p)=
h^{\phi^{n-2}}(t_2)^p$$ implies $1\otimes
y^{p^{n-2}}=h^{\phi^{n-2}}(t_2)$. Comparing with (\ref{equ-algo})
we continue in this way and eventually find a homogeneous element
$t:=t_n\in\theta^{\tn}A$ of degree $s$ such that $1\otimes y=h(t)$
and $1\otimes a_0x=t^{p^n}$. Since $\theta^{\tn}A$ is generated
over $k^{\tn}$ by $1\otimes y$ and $1\otimes a_0x$ the inclusion
$k^{\tn}[t]\subseteq\theta^{\tn}A$ is surjective. Now $t$ is
transcendental over $k^{\tn}$ and there is an isomorphism of
graded $\ell$-algebras
\begin{numequation}\label{equ-trivialization} A_\ell\car
\ell[s^{-1}T], \hskip5pt t\mapsto T\end{numequation} according to
(\ref{equ-FrobII}). In particular, $A\otimes_k A$ is reduced.
Since $c(x)=1\otimes x+x\otimes 1$ we may deduce from $1\otimes
a_0x=t^{p^n}$ that $\bar{c}(t)=1\otimes t+t\otimes 1$ where
$\bar{c}$ denotes the comultiplication in $\theta^{\tn}A$ induced
from $A$. It follows that the above isomorphism is a Hopf
isomorphism when the target has its graded Hopf structure coming from the
additive group with unit element $T=0$.

\vskip8pt

For the choice $s\in \rho(\ell^\times) r$ in $\Gamma$ this yields
graded Hopf algebras $A$ over $k$ which are forms of $B$ with
respect to $\ell/k$.

\vskip8pt

On the other hand, let $A$ be a graded form of $B:= k[r^{-1}T]$
with respect to $\ell/k$ that is also a graded Hopf algebra. Let
$c$ be the comultiplication of $A$ and let $\bar{c}$ be the
induced comultiplication on the base change $\theta^{\tn}A$.
According to (\ref{equ-FrobI}) we have an isomorphism of
$k^{\tn}$-algebras
\begin{numequation}\label{equ-trivializationI}\theta^{\tn}A\car k^{\tn}[r^{-1}T]\end{numequation} which we use as an
identification. Via transport of structure we view the right hand
side as a graded Hopf algebra over $k^{\tn}$. If $t$ denotes a
homogeneous generator of the kernel of the augmentation, then
$t=T-a$ with $a\in (k^{\tn})_r$. By replacing $T$ with $T-a$ we
may therefore assume that $\bar{c}(T)=1\otimes T+T\otimes 1$. Using (\ref{equ-trivializationI}) as an identification, we may
write
\begin{numequation}\label{equ-T}T=\sum_{i} a_i\otimes y_i\end{numequation} with $a_i\in
k^{\tn}$ and $y_i\in A$ where each $a_i\otimes y_i$ is homogeneous
of degree $r$. We may assume that the $a_i$ are a linearly
independent family of homogeneous elements in the graded
$k$-vector space $k^{\tn}$. Then each $y_i\in A$ is homogeneous
with degree, say, $r_i$. Moreover,
$$ \sum_i a_i\otimes c(y_i)=\bar{c}(\sum_{i} a_i\otimes y_i)=\bar{c}(T)=1\otimes T+T\otimes
1=\sum_i a_i\otimes y_i\otimes 1+a_i\otimes 1\otimes y_i$$ implies
$$c(y_i)=1\otimes y_i+y_i\otimes 1$$
for all $i$. 
By the discussion in subsection \ref{subsec-addgroup}, $1\otimes
y_i$ equals under the identification (\ref{equ-trivializationI}) a
$p$-polynomial $f_i\in k^{\tn}[r^{-1}T]$ which is homogeneous of
degree $r_i$.

\vskip8pt

Consider now the graded subring $k[y_1,...,y_n]\subseteq A$
generated by the $y_i$. By (\ref{equ-T}) the inclusion map becomes
surjective after the faithfully flat base change $\phi^n:
k\rightarrow k^{\tn}$ and therefore $k[y_1,...,y_n]=A$. Hence the
graded fraction field $F:={\rm Frac}_\Gamma(A)$ of $A$ is
generated over $k$ by $y_1,...,y_n$. Since $\ell\otimes_k F={\rm
Frac}_\Gamma (\ell [r^{-1}T])$ it follows that ${\rm trdeg}_k(F)=1$
(Lem. \ref{lem-trdeg}) and there is an element, say $y$, among the
homogeneous elements $y_i$ which forms a separating transcendence
basis of $F/k$ (Lem. \ref{lem-septransbasis}). Let $s$ be the
degree of $y\in A$. By our above discussion we have $$1\otimes
y=f(T)\in\theta^{\tn}A=k^{\tn}[r^{-1}T]$$ with a $p$-polynomial
homogeneous of degree $s$
$$f(T)=a_0T+a_1T^p+\cdot\cdot\cdot a_mT^{p^m}\in k^{\tn}[r^{-1}T]$$ with $a_m\neq 0.$
Applying the graded homomorphism $F^{\tn}_A$ to $T\in
k^{\tn}[r^{-1}T]=\theta^{\tn}A$ yields a homogeneous element
$x:=F^{\tn}_A(T)$ of degree $r$ in $A^{\tn}$ that satisfies
$$y^{p^n}=F^{\tn}_A(1\otimes
y)=F^{\tn}_A(f(T))=f(F^{\tn}_A(T))=f(x)$$ in $A^{\tn}$ and
$c(x)=1\otimes x+x\otimes 1.$
\vskip8pt By choice of $y$ the extension of graded fields $F/k(y)$
is separable and so is the extension $F/k(x,y)$. But for
all $i$ we have in $A^{\tn}$ that $y_i^{p^n}=f_i(x)$. We read this
equation in $A$ which implies that $y_i$ is purely inseparable over $k(x)$. 
Since $F=k(y_1,...,y_n)$ this means that $F/k(x,y)$ is purely
inseparable. Consequently, $F=k(x,y)$. We also see that
$A=k[y_1,...,y_n]$ is a graded integral extension of $k[x,y]$.

\vskip8pt

On the other hand, $\theta^{\tn}$ preserves graded direct sums and
so the degree $[F:k(y)]$ equals the degree of the free graded
$\theta^{\tn}k(y)$-module $\theta^{\tn}F$. Since this module
equals ${\rm Frac}_\Gamma(k^{\tn}[r^{-1}T])$ by
(\ref{equ-trivializationI}), this degree equals the monomial
degree in $T$ of the element $1\otimes y=f(T)$ of
$\theta^{\tn}A=k^{\tn}[r^{-1}T]$ which is $p^m$. Since $F=k(x,y)$,
the homogeneous polynomial
$$h(T):=-y^{p^n}+f(T)=-y^{p^n}+a_0T+a_1T^p+\cdot\cdot\cdot +a_mT^{p^m}\in
k(y)[r^{-p^n}T]$$ of degree $s^{p^n}$, which annihilates $x$ and
has monomial degree $p^m$, must be irreducible. The separability
of $x$ over $k(y)$ implies then that $a_0\neq 0$
(\cite[1.14.1]{Ducros}). Since $a_i\in k$ and $y$ is
transcendental over $k$ the graded Gauss lemma (cf. proof of Lem.
\ref{lem-gauss}) shows that $h(T)$ is also irreducible in
$k[y][r^{-p^n}T]$. We thus have an isomorphism of graded Hopf
$k$-algebras
\begin{numequation}\label{equ-trivializationII} k[r^{-p^n}T_1,s^{-1}T_2]/(T_2^{p^n}-f(T_1))\car
k[x,y]\end{numequation} induced by $T_1\mapsto x, T_2\mapsto y$
where the Hopf structure on the source comes from the $2$-dimensional additive group with unit element $T_1=T_2=0$. At this point we
need a lemma.
\begin{lemma}
The graded domain $k[x,y]$ is graded integrally closed.
\end{lemma}
\begin{proof}
We first show that the graded local ring $(k[x,y])_Q$ for any
$Q\in \Spec_\Gamma(k[x,y])$ is integrally closed. Since this graded
local ring is graded noetherian it suffices, according to Lem.
\ref{lem-DVR}, to check that its maximal homogeneous ideal is
generated by a non-nilpotent homogeneous element. To do this let
$R:=k[y]$ and consider the graded ring
extension $$R\rightarrow R[r^{-p^n}T]/(h)=:S.$$ 
Of course, $S\simeq k[x,y]$. The formal derivative of the
polynomial $h$ is the unit $h'=a_0\neq 0$ in $R[r^{-p^n}T]$. The
usual argument proving that so-called standard-\'etale extensions
are \'etale and therefore unramified (e.g. \cite[Thm.
4.2]{KnusOjanguren}) carries over to the graded setting: for the
prime homogeneous ideal $P=Q\cap R$ of $R$ with graded residue
field $\kappa(P)=R_P/PR_P$ the graded $\kappa(P)$-algebra
$\kappa(P)\otimes_R S_Q$ equals a finite product of separable
finite graded field extensions of $\kappa(P)$ and, at the same
time, is a graded local ring because of $PS_Q\subseteq QS_Q$. It
follows that $S_Q/PS_Q$ is a finite and separable graded field
extension of $\kappa(P)$ and, in particular, $P$ generates the
maximal homogeneous ideal of $S_Q$. Since $R$, being isomorphic to
$k[r^{-1}T]$ via $y\mapsto T$, is a graded principal ideal domain,
its local graded rings are graded discrete valuation rings and so
$P$ is generated by a homogeneous non-nilpotent element (Lem.
\ref{lem-DVR}). Thus, $S_Q$ is integrally closed for any $Q$ and
therefore $S\simeq k[x,y]$ is integrally closed, cf. last sentence
in A.2. This proves the lemma.
\end{proof}
Since $A$ is integral over $k[x,y]$ the lemma together with
$F=k(x,y)$ implies that $A=k[x,y].$ The isomorphism
(\ref{equ-trivializationII}) allows us to apply to $A$ our
previous discussion which yields an isomorphism $A_\ell\car
\ell[s^{-1}T]$, cf. (\ref{equ-trivialization}). In view of
(\ref{equ-trivializationI}) we conclude $s\in\rho(\ell^\times)r$
in $\Gamma$. We may summarize the whole discussion in the
following theorem.
\begin{thm}\label{thm-russellgraded}
Let $\ell=k^{p^{-n}}$ and $B:= k[r^{-1}T]$. Let
$s\in\rho(\ell^\times)r$ in $\Gamma$. Let
$$f(T_1)=a_0T_1+a_1T^p+\cdot\cdot\cdot +a_mT^{p^m}\in
k[r^{-p^n}T_1]$$ be a $p$-polynomial homogeneous of degree
$s^{p^n}$ with $a_0\neq 0$. The graded Hopf $k$-algebra
$$A:=k[r^{-p^n}T_1,s^{-1}T_2]/(T_2^{p^n}-f(T_1))$$
is a graded form of $B$ with respect to $\ell/k$. Conversely, any
graded Hopf $k$-algebra, which is a form of $B$ with respect to
$\ell/k$, is of this type.
\end{thm}

Remark: In \cite{Russell} P. Russell determines all forms of the
additive group over a field of positive characteristic. Our above
theorem is a version for graded fields of loc.cit., Thm. 2.1 and
its proof. Following \cite[Part I.\S2]{KMT} we tentatively call
the graded forms appearing in the above theorem {\it of Russell
type}. It is easy to see that a graded version of \cite[Lem.
1.1.i]{Russell} holds: any $A$ as above admits a trivialization
$A_{\ell'}\car B_{\ell'}$ which is defined over a {\it finite}
subextension $\ell'/k$ inside $\ell$.

\vskip8pt

To determine how many graded forms of Russell type exist, one
needs to know when such a form is trivial and when two such forms
are isomorphic. In the ungraded case there is a very precise
answer to these two questions (loc.cit., Cor. 2.3.1 and Thm. 2.5).
In the presence of gradings the arguments of loc.cit. that lead to
this answer seem to allow no direct generalization. For example,
the Frobenius map $T\mapsto T^p$ is {\it not} a graded degree $1$
endomorphism of $B$ if $r\neq 1$. Hence, at this point we will say
something only in case $r=1$ and leave the general case for future
work.

\vskip8pt

In case $r=1$ the grading $\rho_A$ of a form $A$ of Russell type
has image in $\rho(k^\times)$. Hence the graded algebra $A$ is
{\it induced} from $A_1$ in the sense of graded Frobenius
reciprocity \cite[Thm. 2.5.5]{NastasescuVO}. In particular, the
whole situation is governed by the homogeneous parts of degree
$1$. To be more precise, any $p$-polynomial $f(T)\in k[T]$
homogeneous of degree $1$ lies in $k_1[T]$ and uniquely
corresponds to an endomorphism
of $\mathbb{G}_{a,k_1}$, the additive group over $k_1$ (subsection \ref{subsec-addgroup}). 
Moreover, $\ell_1=(k_1)^{p^{-n}}.$ Recall that $k_1[F]^{(n)}$
equals the image of $k_1[F]$ under the map
$\tau\mapsto\tau^{(n)}$. Let $B=k[T]$ as graded $k$-algebra (i.e.
$r=1$).

\begin{cor}
Let $f(T_1)=a_0T_1+a_1T^p+\cdot\cdot\cdot +a_mT^{p^m}\in k[T_1]$
be a $p$-polynomial homogeneous of degree $1$ with $a_0\neq 0$ and
let $\tau$ be its endomorphism of $\mathbb{G}_{a,k_1}$. The graded
Hopf $k$-algebra
$$A(\tau):=k[T_1,T_2]/(T_2^{p^n}-f(T_1))$$
is a graded form of $B$ with respect to the extension
$\ell=k^{p^{-n}}$ of $k$. The form $A(\tau)$ is trivial if and
only if $\tau c\in k_1[F]^{(n)}$ for some $c\in k_1^\times$. In
particular, $k_1$ is perfect if and only if all $A(\tau)$ are
trivial.
\end{cor}
\begin{proof}
Passing to the homogeneous part of degree $1$ induces a map from
the set of isomorphism classes of the $A(\tau)$ to the set of
forms of $\mathbb{G}_{a,k_1}$ with respect to the extension of
characteristic $p$-fields $k_1\subseteq\ell_1$. The map is
surjective by \cite[Thm 2.5]{Russell} and injective by graded
Frobenius reciprocity, cf. Lem. \ref{lem-extension}. Then
\cite[Cor. 2.3.1]{Russell} gives the first statement. If $k_1$ is
perfect, then $k_1[F]=k_1[F]^{(n)}$. If $k_1$ is not perfect,
there is a $p$-polynomial whose coefficient $a_1$ at $T^p$ is not
a $p$-power in $k_1$. If $\tau$ is the corresponding endomorphism,
then $\tau c\notin k_1[F]^{(1)}$ for any $c\in k_1^\times$ and
hence $A(\tau)$ is nontrivial.
\end{proof}
Recall that $H^1(\ell/k, {\bf \Aut}^{\rm gr}\; B)$ equals the
pointed set of graded forms of $B$ with respect to $\ell/k$.
Recall that the group $U_n$ equals the image of $k[F]^*$ under the
projection $k[F]\rightarrow k[F]/k[F]F^n$. Let
$$G_n:=U_n\times k_1^\times$$ be the direct product of the multiplicative groups
$U_n$ and $k_1^\times$. There is an action of $G_n$ on the pointed
set $U_n$ given by
$$ (\bar{\sigma},c). \bar{\tau}:= (\sigma^{(n)}\tau c^{-1}) {\rm
~mod~} (F^n).$$ Here, $\bar{\sigma}=\sigma {\rm ~mod~}(F^n)$ is
the class of $\sigma\in k[F]^*$ in $U_n\subseteq k[F]/k[F]F^n$ and
similarly for $\tau$. Moreover, $\sigma^{(n)}$ equals the image of
$\sigma$ under the extension of $\phi^n$ to $k[F]$. Let $U_n/G_n$
be the pointed set of orbits corresponding to this action.

\vskip8pt

Example: In loc.cit., p. 534, there is the following description
of the set $U_n/G_n$ in the simplest case $n=1$. Let $W_0$ be a
complementary $\bbF_p$-subspace to $k_1^p\subseteq k_1$ (in case
$p\neq 1$) and for each $i\geq 1$ let $W_i$ be a copy of $W_0$.
Then $U_1=k_1^\times$ acts linearly on the space $W=\oplus_{i\geq
1} W_i$ by $c.(\sum_i a_i):= \sum_i c^{p(1-p^{i})}a_i$. Mapping
$\sum_i a_i$ to $1+\sum_i a_iF^{i}\in k[F]^*$ induces an
isomorphism of pointed sets
$$W/k^\times\car U_1/G_1.$$

\vskip8pt

 Back in the case of general $n\geq 1$, the proof of the
preceding corollary and loc.cit., Thm. 2.5 imply the following
corollary.
\begin{cor}\label{cor-orbits}
Let $\ell=k^{p^{-n}}$ and $r=1$. The map $$ k_1[F]^*\rightarrow
H^1(\ell/k, {\bf \Aut}^{\rm gr}\; B), \hskip10pt\tau\mapsto {\rm
isomorphism~class~of~}A(\tau)$$ induces an inclusion between
pointed sets
$$U_n/G_n\hookrightarrow H^1(\ell/k, {\bf \Aut}^{\rm gr}\; B).$$
\end{cor}

\subsection{A class of wildly ramified forms}

We let $\Gamma=\bbR^\times_+$ and let $k$ be a nonarchimedean
field which is complete with respect to a nontrivial absolute
value $|.|$ and whose residue field $\tk_1$ has characteristic
$p>0$.

\vskip8pt

Let $n\geq 1$. Let $k\subseteq\ell$ be a complete field extension
whose absolute value restricts to the one on $k$ and which has the
property $(\tgk)^{p^{-n}}\subseteq\tgell$.


\vskip8pt

Let $r\in\Gamma$ and $B:=k\{r^{-1}T\}$. 
Let $s\in\Gamma$ and
$\tilde{f}(T_1)=a_0T_1+a_1T^p+...+a_mT^{p^m}\in\tgk[r^{-p^n}T_1]$
be a homogeneous $p$-polynomial with $a_0\neq 0$ of degree
$s^{p^n}$. Let $\tgA$ be the form of $\tgB=\tgk[r^{-1}T]$
corresponding to these data via the theorem above. Let $f(T_1)$ be
a homogeneous lift of $\tilde{f}$ in $k\{r^{-p^n}T_1\}$ and
consider the affinoid $k$-algebra
$$ A:= k\{r^{-p^n}T_1,s^{-1}T_2\}/(T^{p^n}_2-f(T_1)).$$
The quotient norm on $A$ gives a filtration whose graded ring
equals $\tgA$ \cite[Cor. 1.1.9/6]{BGR}. Since $\tgA$ is reduced,
the quotient norm is power-multiplicative and, thus, equals the
spectral seminorm (loc.cit., Prop. 6.2.3/3). Hence, this seminorm
is a norm and $A$ is reduced. Also, $\tgA$ is the graded reduction
of the affinoid algebra $A$. Since $\tgA$ is a form of $\tgB$ the
graded reduction of $A_\ell=\ell\hat{\otimes}_k A$ equals
$\tgell\otimes_{\tgk}\tgA\simeq \tgell[r^{-1}T]$ according to Lem.
\ref{lem-keyreduced}. Let $\tilde{f}$ be the preimage of $T$ under
the latter isomorphism and $f$ a homogeneous lift in $A_\ell$.
Mapping $T$ to $f$ induces a homomorphism $\psi:
\ell\{r^{-1}T\}\rightarrow A_\ell$ of $\ell$-affinoid algebras
whose graded reduction is an isomorphism. If $\ell'$ is any
complete extension field of $k$ containing $\ell$ then
$\tilde{\ell'}_\bullet\otimes_{\tgk}\tgA$ is reduced since $\tgA$
is a form of $\tgB$. Lem. \ref{lem-gradediso} yields now that
$\psi$ is an isomorphism and so $A$ is a form of $B$ with respect
to $k\subseteq\ell$. We summarize:

\begin{prop}
Any form of $\tgk[r^{-1}T]$ of Russell type with respect to
$\tgk\subseteq\tgell$ lifts to a wildly ramified form of
$k\{r^{-1}T\}$ with respect to $k\subseteq\ell$ and this map is
injective on isomorphism classes of forms.
\end{prop}
Let $r=1$. Consider the pointed set of orbits $U_n/G_n$ from the
preceding subsection. Composing the inclusion of the preceding
proposition with the one from Cor. \ref{cor-orbits} yields the
following corollary.
\begin{cor}
In case $r=1$, there is a canonical inclusion of pointed sets
$$U_n/G_n\hookrightarrow H^1(\ell/k, {\bf \Aut}\hskip2pt B).$$
\end{cor}

\section{Picard groups and $p$-radical descent}\label{descent}

All rings are commutative and unital. All ring homomorphisms are
unital.
\subsection{Krull domains}\label{krull}

We recall some divisor theory of Krull domains \cite[VII.1]{B-CA} thereby fixing some notation. Let $A$ be an integral domain
with quotient field $K$. The ring $A$ is called a {\it Krull
domain} if there exists a family of discrete valuations
$(v_i)_{i\in I}$ on $K$ such that $A$ equals the intersection of
the valuation rings of the $v_i$ and such that for $x\neq 0$
almost all $v_i(x)$ vanish. 
A noetherian Krull domain of dimension $1$ is the same as a
Dedekind domain (loc.cit., Cor. VII.1.3).

\vskip8pt

Let $A$ be a Krull domain and $P(A)$ the set of height $1$ prime
ideals of $A$. The free abelian group $D(A)$ on $P(A)$ is called
the {\it divisor group} of $A$. Given $P\in P(A)$ the localization
$A_P$ of $A$ is a discrete valuation ring. Let $v_P$ be the
associated valuation. The {\it divisor map}
\[\dv_A: K^\times\longrightarrow D(A)~,~~~ x\mapsto\sum_{P\in P(A)}
v_P(x)P\] is then a well-defined group homomorphism giving rise to
the {\it divisor class group} $$\Cl(A):=D(A)/{\rm Im~} \dv_A.$$

\vskip8pt

If $A'\subseteq A$ is a subring which is a Krull domain itself and
if $P'\in P(A'),~ P\in P(A)$ are height $1$ prime ideals with
$P\cap A'=P'$ let $e(P/P')\in\bbN$ denote the ramification index
of $P$ over $P'$. Then $j(P'):=\sum e(P/P')P$ induces a
well-defined group homomorphism \[j: D(A')\rightarrow D(A)\] where
the sum runs through all $P\in P(A)$ with $P\cap A'=P'$. If
$A'\subseteq A$ is an {\it integral} ring extension, then $j$
factors through a group homomorphism
\[\overline{\jmath}: \Cl(A')\rightarrow \Cl(A)\] \cite[Prop. VII.\S1.14]{B-CA}.

\subsection{$p$-radical descent}\label{radical}

Let $A$ be for a moment an arbitrary commutative ring and $m\geq
0$ a natural number. Recall that a {\it higher derivation}
$\partial$ of {\it rank} $m$ on $A$ is an ordered tuple
$$\partial=(\partial_0,...,\partial_m)$$ of additive endomorphisms $\De_k$ of
$A$ satisfying the convolution formula
\begin{equation}\label{conv}
\partial_k(ab)=\sum_{j=0,...,k}\partial_j(a)\partial_{k-j}(b)
\end{equation}
for all $a,b\in A$ and such that $\De_0={\rm id}_A$
\cite[IV.9]{Jacobson}. Let
\[A[T]_m:=A[T]/(T^{m+1})\] be the $m$-truncated polynomial ring over
$A$ with its augmentation map $$\epsilon: A[T]_m\longrightarrow
A,\hskip5pt f\mapsto f {\rm~ mod~} (T).$$The higher derivation
$\partial$ induces an injective ring homomorphism
\[A\rightarrow A[T]_m,~a\mapsto\sum_{j=0,...,m}\partial_j(a)T^j\]
which we again denote by $\partial$. Obviously,
$\epsilon\circ\partial ={\rm id}_A$. Conversely, any ring
homomorphism $\partial: A\rightarrow A[T]_m$ with the property
$\epsilon\circ\partial ={\rm id}_A$ induces in an obvious way a
higher derivation of rank $m$ on $A$. An element $a\in A$ such
that $\partial(a)=a$ is called a {\it constant}. The constants
form a subring of $A$. We say that $\partial$ is {\it nontrivial}
if there exists an element in $A$ which is not a constant. In this
case, we define the {\it order} $\mu(\partial)$ of $\partial$ as
$$ \mu(\partial):=\min \{ 1\leq j \leq m \hskip3pt |\hskip3pt \De_j\neq 0\}.$$

\vskip8pt

We assume in the following that the ring $A$ has characteristic
$p>0$. In this case we introduce for a nontrivial $\partial$ the
{\it exponent} $n(\partial)$ as
$$n(\partial):=\min \{n\hskip3pt | \hskip3pt m<\mu(\partial)\cdot p^{n}\}.$$
Note that $n(\partial)\geq 1$.
\begin{lemma}\label{lem-constants}
Let $\partial$ be nontrivial. Any power $a^{p^{n(\partial)}}$ with
$a\in A$ is a constant.
\end{lemma}
\begin{proof}
Suppose $\De_j(a)\neq 0$ for some $j\geq 1$. Then
$\mu(\partial)\leq j$ and hence $m<j\cdot p^{n(\partial)}$. It
follows that $(\partial_j(a)T^j)^{p^{n(\partial)}}=0$ in $A[T]_m$.
Since $\partial$ is a homomorphism and $A[T]_m$ has characteristic
$p$ we obtain
$$\partial(a^{p^{n(\partial)}})=\partial(a)^{p^{n(\partial)}}=\sum_{j=0,...,m}(\partial_j(a)T^j)^{p^{n(\partial)}}=a^{p^{n(\partial)}}.$$
\end{proof}

 \vskip8pt

If $C\subseteq A$ is a subring with $\De_j(C)\subseteq C$ for all
$j=0,..,m$ the ring $C$ is called {\it invariant}. We obtain in
this case a higher derivation on $C$ of the same rank.

\vskip8pt

In \cite{Baba} K. Baba uses the notion of higher derivation to
build up a $p$-radical descent theory of higher
exponent. We recall that the case of exponent one using ordinary
derivations is classical and due to P. Samuel \cite{Samuel2}.
We shall
only need a special case of Baba's theory as follows. 
Consider a Krull domain $A$ of characteristic $p>0$ with quotient
field $K$ and a nontrivial higher derivation $\De$ of rank $m$ on
$K$ such that the subring $A\subseteq K$ is invariant. Let
$K'\subseteq K$ be the field of constants and let $A':=A\cap K'$.
Then $K'$ is the quotient field of $A'$. If $\{v_i\}_{i\in I}$ is
a family of discrete valuations on $K$ exhibiting $A$ as a Krull
domain, their restrictions to $K'$ prove $A'$ to be a Krull
domain.

By the previous lemma $A'\subseteq A$ is an integral ring
extension. We therefore have the canonical homomorphisms $j:
D(A')\rightarrow D(A)$ and $\overline{\jmath}: \Cl(A')\rightarrow
\Cl(A)$. On the other hand, we have inside the group of units
$(K[T]_m)^\times$ the subgroup of so-called {\it logarithmic
derivatives}
$$\cL:=\{\frac{\De(z)}{z} \hskip2pt | \hskip2pt z\in K^\times\}$$ as well
as its subgroups
\begin{equation}\label{integrallog}\begin{array}{ll}
  \cL_A:=& \cL\hskip2pt \cap\hskip2pt (A[T]_m)^\times,\ \\
  &  \\
  \cL'_A:=& \{\frac{\De(u)}{u}\hskip2pt | \hskip2pt u\in A^\times\} \\
\end{array}
\end{equation} with $\cL'_A\subseteq\cL_A$.
\begin{lemma}\label{lem-log}
The abelian group $\cL$ has exponent $p^{n(\partial)}$ and one has
$$\cL=\{ \frac{\De(z)}{z}\hskip2pt | \hskip2pt z\in A-\{0\}\}.$$
\end{lemma}
\begin{proof}
If $z\in K^\times$ then
$\partial(z)^{p^{n(\partial)}}=\partial(z^{p^{n(\partial)}})=z^{p^{n(\partial)}}$
by the preceding lemma. This proves the first claim. Now consider
$\frac{\De(z)}{z}$ with $z\in K^\times.$ Write $z=f/g$ with
nonzero $f,g\in A$ and let $m:=p^{n(\partial)}-1$. As we have just
seen $(\partial(g)/g)^m=(\partial(g)/g)^{-1}$ in $(K[T]_m)^\times$
and therefore
$$\partial(fg^m)/(fg^m)=(\partial(f)/f)\cdot
(\partial(g)/g)^m=\partial(f/g)/(f/g).$$ Since $fg^m\in A-\{0\}$
this yields the assertion.
\end{proof}

\vskip8pt

The fundamental construction in \cite{Baba}, (1.5) produces an
injective homomorphism
\[\Phi_A: \ker\overline{\jmath}\longrightarrow \cL_A/\cL'_A,\hskip10pt D{\rm~
mod~}~\dv_{A'}(K'^\times)\mapsto \frac{\De(x_D)}{x_D} {\rm
~mod~}\cL'_A\] where $x_D\in K^\times$ is chosen such that
$$\dv_{A}(x_D)=j(D).$$ We consider the following
condition on the tuple $(A,\partial)$:
\[\begin{array}{lr} (\heartsuit)
 & ~~~~~[K:K']=p^{n(\De)}~~{\rm and}~~\partial_{\mu(\partial)}(a)\in A^\times~{\rm for~some~} a\in A. \\
   \\
\end{array}\]

\begin{prop}\label{prop-Baba}
If $(\heartsuit)$ holds, then $\Phi_A$ is an isomorphism.
\end{prop}
\begin{proof}
This is a special case of loc.cit., Thm. 1.6. Note that we have
$r=1$ in the notation of loc.cit.
\end{proof}
Thus, if $(\heartsuit)$ holds and the class group of $A$ vanishes,
then the class group of the ring of constants $A'$ has exponent
$p^{n(\partial)}$ and admits a fairly explicit description in
terms of logarithmic derivatives in the ring $A[T]_m$. This will
be our main application.

\vskip8pt

We shall need some information how the property $(\heartsuit)$
behaves when the ring of constants varies. To this end we consider
an injective homomorphism $A'\rightarrow B$ into an integrally
closed domain $B$ such that $A_B:=B\otimes_{A'} A$ is an integral
domain. Let $L'=\Frac(B)$ be the fraction field of $B$. Since $A_B$
is integral over $B$, we then have
$$L:=\Frac(A_B)=L'\otimes_{K'} K$$ for the fraction field of $A_B$
\cite[Lem. \S9.1]{Matsumura}. Since $B$ is integrally closed,
one has $L'\cap A_B=B$.
\begin{lemma}\label{lem-base}
Suppose that $A$ is flat as a module over $A'$ and that the
$A'$-submodule ${\rm Im}\hskip2pt (\partial_0-\partial) \subseteq
A[T]_m$ associated with the homomorphism $\partial: A\rightarrow
A[T]_m$ is flat. The field $L$ admits a higher derivation
$$\De_L:={\rm id}_{L'}\otimes_{K'}\partial$$ of the same rank and
order as $\De$ such that the subring $A_B$ is invariant. The field
of constants with respect to $\partial_L$ equals $L'$.
\end{lemma}
\begin{proof}
The free $A$-module $A[T]_m$ is flat as an $A'$-module whence ${\rm
Tor}_1^{A'}(B,A[T]_m)=0$. So tensoring the short exact sequence
$$ 0\longrightarrow (T^{m+1})\longrightarrow A[T]\longrightarrow
A[T]_m\longrightarrow 0$$ over $A'$ with $B$ shows that
$B\otimes_{A'} A[T]_m\simeq (A_B)[T]_m$. We therefore obtain a
higher derivation $A_B\rightarrow (A_B)[T]_m$ by $\partial_B:={\rm
id}_B \otimes_{A'} \partial$ of rank $m$. According to formula
(\ref{conv}) all $\partial_k$ are $A'$-linear whence all
$(\partial_B)_k={\rm id}_B\otimes_{A'}\partial_k$ are $B$-linear.

Let $N:={\rm Im}\hskip2pt (\partial_0-\partial)$. By assumption we have
${\rm Tor}_1^{A'}(B,N)=0$. So tensoring the exact sequence of
$A'$-modules
$$ 0\longrightarrow A'\stackrel{\subseteq}{\longrightarrow}
A\stackrel{\partial_0-\partial}{\longrightarrow} N\longrightarrow
0$$ with $B$ we obtain that $B$ equals the ring of constants for
$\partial_B$ on $A_B$. Let $S:=B-\{0\}$. Then
$L=S^{-1}A_B=K\otimes_{K'} L'$. The ring homomorphim $\partial_B$
is $B$-linear whence $\partial_L:=S^{-1}\partial_B$ is an
$L'$-linear ring homomorphism $L\rightarrow L[T]_m$ and a higher
derivation of $L$ of rank $m$. We have $(\partial_L)_k
=S^{-1}(\partial_B)_k= L'\otimes_L
\partial_k$ for all $k=0,...,m$ whence
$\mu(\partial_L)=\mu(\partial)$.  Moreover, $x/s\in L$ with $x\in
A_B, s\in S$ is constant with respect to $\partial_L$ if and only
if $x$ is constant with respect to $\partial_B$, i.e. if and only
if $x\in B$. This shows that $L'$ equals the field of constants
with respect to $\partial_L$.

\end{proof}
\begin{cor}\label{cor-baseexp}
Keep the assumptions of the preceding lemma. Suppose that the ring
$A_B$ is again a Krull domain. If $(A,\De)$ satisfies
$(\heartsuit)$ then $(A_B,\De_L)$ satisfies $(\heartsuit)$. In
this case, $B$ is a Krull domain.
\end{cor}
\begin{proof}
Since $\mu(\partial_L)=\mu(\partial)$ one has
$[L':L]=[K':K]=p^{n(\partial)}=p^{n(\partial_L)}$. Take $a\in A$
such that $\partial_{\mu(\partial)}(a)\in A^\times$. It follows
that $\partial_{\mu(\partial_L)}(1\otimes a)= 1\otimes
\partial_{\mu(\partial)}(a)\in (A_B)^\times$. This shows $(\heartsuit)$ for $(A_B,\partial_L)$.
We have explained above that in this case, $B=L'\cap A_B$ is again
a Krull domain.
\end{proof}

We apply these results to the following special case
\cite[IV.9]{Jacobson}. Let $k$ be a field of characteristic $p>0$
and let $A:=k[t^{\pm 1}]$ be the ring of Laurent polynomials in
one variable $t$ over $k$ and let $K$ be its fraction field. For
any $m\geq 0$ we have the ring homomorphism
$$A\rightarrow A[T]_m, \hskip5pt t\mapsto {\rm class~of~}t+T$$ which is a higher derivation $\partial$ of rank $m$.
Let $m'\geq 0$ and consider $\partial$ for $m:=p^{m'}-1$. Let
$A':=k[t^{\pm p^{m'}}]$ and write $K'$ for its quotient field. We
have
$$\partial(t^{p^{m'}})=\partial(t)^{p^{m'}}=(t+T)^{p^{m'}}=t^{p^{m'}}$$
and therefore $\partial$ is $A'$-linear. Since $K'[t]=K$ we obtain
by localization at $A'-\{0\}$ a $K'$-linear higher derivation
$$\partial: K\rightarrow K[T]_m$$ of rank $m$ that maps $t$ to (the class of) $t+T$.
Note that for $j=0,...,m$ each $\partial_j$ is a $K'$-linear
endomorphism on $K$ whose effect on the $K'$-basis $1,t,...,t^{m}$
of $K$ is given by
\begin{numequation}\label{equ-binomial}\De_jt^{i}={i\choose j}t^{i-j}\end{numequation} where
${i\choose j}=0$ if $j>i$. This is obvious from the binomial
expansion of $\partial(t^{i})=(t+T)^{i}$.
\begin{lemma}\label{lem-hypothesis}
The nontrivial higher derivation $\partial$ has the invariants
$\mu(\De)=1$ and $n(\De)=m'$ and the subring $A$ is invariant. The
field of constants equals $K'$ with $K'\cap A=A'$. The ordered pair
$(A,\De)$ satisfies $(\heartsuit)$.
\end{lemma}
\begin{proof} First of all, $[K:K']\leq p^{m'}$.
From $\partial(t)=t+T$ we deduce $\partial_1(t)=1$ whence
$\mu(\partial)=1$. The identity $n(\De)=m'$ is then obvious. Let
$K''$ be the field of constants of $\partial$. 
We have $$K^{p^{m'}}\subseteq K' \subseteq K''\subseteq K$$ and so
the field extension $K/K''$ is purely inseparable. Hence the
minimal polynomial of $t$ over $K''$ has the form $X^{p^n}-a$ with
some $a\in K''$ and some $n$. Suppose $n<m'$ so that $p^n\leq
p^{m'}-1=m$. Then $t^{p^n}=a\in K''$ is a constant. But
(\ref{equ-binomial}) implies $\De_{p^n}(t^{p^n})=1$, a
contradiction. Hence $[K:K'']=p^{m'}$ and our initial remark
implies $K'=K''$. Since $A'$ is integrally closed, we have $K'\cap
A=A'$. Finally, $[K:K']=p^{m'}=p^{n(\De)}$ and $\De_1(t)=1\in
A^\times$. Thus the ordered pair $(A,\De)$ satisfies $(\heartsuit)$.
\end{proof}

For any $m'\geq 0$ we call the homomorphism $\partial:
K\rightarrow K[T]_m$ induced by $t\mapsto t+T$ {\it the standard
higher derivation over $k$ of exponent $m'$}. It has rank
$m=p^{m'}-1$.

\section{Affinoid algebras and Quillen's theorem}


We let $\Gamma=\bbR^\times_+$ and let $k$ be a nonarchimedean
field which is complete with respect to a nontrivial absolute
value $|.|$. We assume that $|.|$ is a {\it discrete} valuation,
i.e. $|k^\times|=|\pi|^{\bbZ}$ for some uniformizing element
$\pi\in k$ with $0<|\pi|<1$. We then have
$$\tgk \car \tk_1[t^{\pm 1}]$$ by mapping the principal symbol
of $\pi$ to the variable $t$. In particular, the ring $\tgk$ is
noetherian. Let $p$ be the characteristic exponent of the residue
field $\tk_1$ of $k$.

\vskip8pt

Let $A$ be a strictly affinoid algebra over $k$ which is an
integral domain. Let $d$ be the dimension of $A$. By the Noether
normalization lemma \cite[Cor. 6.2.2/2]{BGR} there is an
injective homomorphism $k\{T_1,...,T_d\}\hookrightarrow A$ which
is finite and strict with respect to spectral norms. By loc.cit.,
Prop. 6.2.2/4, there is a number $c\geq 1$ such that $|A|\subseteq
|k|^{1/c}$. Let $q=|\pi|^{1/c}$. We then have a decreasing
complete and exhaustive $\bbZ$-filtration on $A$

$$\cdot\cdot\cdot\subseteq F_{s+1}A\subseteq F_sA\subseteq\cdot\cdot\cdot$$

by the additive subgroups
\[
  F_s A:=\{a\in A,~|a|\leq q^{-s}\}=A_{q^{-s}}\]
for  $s\in\bbZ$. Letting $gr_s A:=F_sA/F_{s+1}A$ we have
\[\tgA=\oplus_{s\in\bbZ}~gr_s A\] for the graded reduction of the affinoid algebra $A$.
The ring homomorphism $\tgk\{T_1,...,T_d\}\hookrightarrow \tgA$ is
finite according to \cite[Prop. 3.1(ii)]{TemkinII}. In particular,
the ring $\tgA$ is noetherian.

\vskip8pt

Besides $\tgA$ we need to introduce the {\it Rees ring}
$R(A)_\bullet$ of the filtered ring $A$ \cite{LVO}. To this end
let $A[X^{\pm 1}]$ be the Laurent polynomial ring over $A$ in a
variable $X$ with its $\bbZ$-gradation by degree in $X$. Then
$R(A)_\bullet$ equals the subring
$$ R(A)_\bullet=\oplus_{s\in\bbZ} (F_s\;A)X^s\subseteq A[X^{\pm 1}]$$
with its induced gradation. Write $(R(A)_\bullet)_{(X)}$ for its
localization at the multiplicatively closed subset
$\{X,X^2,...\}.$ One has the identities

\begin{numequation}\label{equ-rees} (R(A)_\bullet)_{(X)}\simeq A[X^{\pm 1}] \hskip10pt {\rm and}\hskip10pt \RA/X\RA\simeq
\tgA.\end{numequation}

Since $A$ is reduced, the filtration $F_\bullet A$ on $A$ is
complete. Since $\tgA$ is noetherian, so is $R(A)_\bullet$
according to \cite[Prop. II.1.2.3]{LVO}.

\vskip8pt

Recall that a commutative noetherian ring $R$ is called {\it
regular} if all its localizations at prime ideals are regular
local rings. By the Auslander-Buchsbaum formula \cite[Thm.
3.7]{AuslanderBuchsbaum} this is equivalent to $R$ having finite
{\it global dimension}, i.e. there is a number $0\leq d <\infty$
such that any finitely generated module has a projective
resolution of length $\leq d$. In this case, the global dimension
${\rm gld}(R)$ of $R$ is the smallest of such numbers $d$ and
coincides with the (Krull) dimension of $R$.
\begin{prop}
Suppose that $\tgA$ is a regular ring. Then $A$ and $\RA$ are
regular rings with ${\rm gld}(A)\leq {\rm gld}(\tgA)$ and ${\rm
gld}(\RA)\leq 1+{\rm gld}(\tgA)$.
\end{prop}
\begin{proof}
This is the commutative case of the main result of
\cite{LVOGlobalAuslanderRees}.
\end{proof}
The regularity of $\RA$ implies that any finitely generated graded
$\RA$-module has a finite resolution by finitely generated graded
projective $\RA$-modules with degree zero morphisms
\cite[6.5]{MCR}.

\vskip8pt

Let $K_0(A)$ and $K_0(\tgA)$ be the Grothendieck group of the ring
$A$ and $\tgA$ respectively. We have the following generalization
of the $K_0$-part of Quillen's theorem \cite{QuillenH} which is
due to Li Huishi, M. van den Bergh and F. van Oystaeyen
\cite[Cor. 2.6]{LVOVB}.
\begin{prop}
Suppose $\tgA$ is regular. There is a group isomorphism
\begin{equation}\label{homo}
\nu:
K_0(A)\stackrel{\cong}{\longrightarrow}K_0(\tgA)\end{equation}
such that $\nu([A])=[\tgA]$.
\end{prop}
Let us explain where the isomorphism comes from. It is defined via
the following commutative diagram

\[\xymatrix{
  K_{0g}(\RA)  \ar[r]^>>>>>{\beta} \ar[d]_{\simeq}^{i} &K_0(A)   \ar[d]_{\simeq}^{\nu} \\
  K_{0g}(\tgA)\ar[r]^{\gamma} & K_0(\tgA)  }\] where the horizontal maps are surjective and the
vertical maps are isomorphisms. Here, $K_{0g}$ equals $K_0$
applied to the category of finitely generated projective graded
modules over a $\bbZ$-graded ring with degree zero graded
morphisms. Note here that a projective object in
the graded category is the same as a graded module which is
projective after forgetting gradations \cite[Prop.
7.6.6]{MCR}. The upper horizontal map comes from the homogeneous
localization $\RA\rightarrow(\RA)_{(X)}=A[X^{\pm 1}]$ and the
equivalence of categories $M\mapsto A[X^{\pm 1}]\otimes_A M$
between finitely generated $A$-modules and finitely generated
$\bbZ$-graded $A[X^{\pm 1}]$-modules. Note that a quasi-inverse to this latter equivalence is
given by passing from a $\bbZ$-graded $A[X^{\pm 1}]$-module to its
homogeneous part of degree zero. The map $\gamma$ comes from the
functor "forgetting the gradation". The map $i$ is induced from
the canonical map $\RA\rightarrow\RA/X\RA=\tgA$ and the map $\nu$
is defined so as to make the diagram commutative. Let us make the
map $\nu$ completely explicit. Let $P$ be a finitely generated
projective $A$-module and $P':=A[X^{\pm 1}]\otimes_A P$ the
corresponding $\bbZ$-graded $A[X^{\pm 1}]$-module. Choose a
finitely generated graded $\RA$-submodule $P''$ inside $P'$ that
generates $P'$ over $A[X^{\pm 1}]$.
Since $\RA$ is regular, there is a finite resolution by finitely
generated projective graded $R(A)_\bullet$-modules and degree zero
morphisms
$$ 0\longrightarrow P_n\longrightarrow
P_{n-1}\longrightarrow\cdot\cdot\cdot\longrightarrow
P_0\longrightarrow P''\longrightarrow 0.$$ Then
$$\nu([P])=\sum_{i=0,...,n}(-1)^{i}\hskip3pt [P_i/XP_i].$$
using $\RA/X\RA=\tgA$.

\vskip8pt

Let $\Pic(R)$ be the {\it Picard group} of a commutative ring $R$,
i.e. the abelian group (with tensor product) of isomorphism
classes of finitely generated projective $R$-modules of rank $1$.
\begin{prop}
Let $A$ be of dimension $1$ with $\tgA$ regular. The natural
homomorphism $$\Pic(\RA)\longrightarrow \Pic((\RA)_{(X)})$$ is
surjective.
\end{prop}
\begin{proof}
To ease notation write $S:=\RA$. Since $S$ is regular, all its
localizations at prime ideals are factorial rings by the
Auslander-Buchsbaum-Nagata theorem \cite{AuslanderBuchsbaumII}.
Since $S$ is an integral domain, $X$ is not a zero divisor and
therefore \cite[Prop. 7.17]{Bass} implies the claim. We briefly
recall from loc.cit. how one produces preimages. Let $P\in
\Pic(S_{(X)})$. Since $S$ is an integral domain, $P$ is isomorphic
to an invertible ideal $I\subseteq S_{(X)}$. Let $I_0=I\cap S$.
The bidual $(I_0)^{**}={\rm \Hom}_{S}({\rm \Hom}_{S}(I_0,S),S)$ is a
reflexive $S$-module and lies in $\Pic(S)$. Since $I$ is a
reflexive $S_{(X)}$-module, $(I_0)^{**}$ localizes to $I$.
\end{proof}
Remark: If in the situation of the proposition we have that $\tgA$
is an integral domain, then the surjection on Picard groups is
even a bijection. Indeed, in this case $X\in R(A)_\bullet$
generates a prime ideal according to (\ref{equ-rees}) and then
\cite[Prop. III.7.15]{Bass} yields the assertion.

\vskip8pt

Back in the general case, let $M,N$ be finitely generated
projective $A$-modules of rank $m,n$ respectively. The formula
$$\wedge^{m+n}(M\oplus
N)=\sum_{k=0,...,m+n}\wedge^{k}(M)\otimes\wedge^{m+n-k}(N)=\wedge^{m}(M)\otimes\wedge^n(N)$$
shows that $[M]\mapsto [\wedge^m(M)]$ defines a group homomorphism
$\det: K_0(R)\rightarrow \Pic(R)$. Let ${\rm rk}(M)\in\bbN$ denote
the rank of a finitely generated projective $A$-module $M$ (recall
that $A$ is an integral domain). We have the short exact sequence

\begin{equation}\label{equ-ss} 0\longrightarrow SK_0(A)\longrightarrow K_0(A)\stackrel{{\rm
rk}\oplus\det}{\longrightarrow} \bbZ\oplus \Pic(A)\longrightarrow
0.\end{equation}

The group $SK_0(A)$ consists of classes $[P]-[A^m]$ where $P$ has
some rank $m$ and $\wedge^m P\cong A$. If $A$ has dimension $1$,
then, since $A$ is noetherian, $SK_0(A)=0$ by Serre's theorem
\cite[Thm. IV.2.5]{Bass}. We therefore have
$K_0(A)\simeq\bbZ\oplus \Pic(A)$ in this case.

\begin{prop}\label{prop-comparisonPic}
Let $A$ be of dimension $1$ such that the ring $\tgA$ is regular.
There is a commutative diagram of abelian groups
\[\xymatrix{
0\ar[r]&  \bbZ\ar[r] \ar[d]^{=} &  K_0(A) \ar[r]^>>>>>{\det} \ar[d]_{\simeq}^{\nu} &  \Pic(A) \ar[r] \ar[d]_{\simeq}^{\bar{\nu}} & 0\\
0\ar[r]&  \bbZ \ar[r] & K_0(\tgA) \ar[r]^>>>>>{ \det} & \Pic(\tgA)
\ar[r]& 0 }\] in which rows are exact and vertical maps are
isomorphisms. The rank map induces a canonical splitting in both
rows.
\end{prop}
\begin{proof}
Since $A$ is commutative, the map $1\mapsto[A]$ defines an
inclusion $\bbZ\hookrightarrow K_0(A)$ and similarly for $\tgA$.
In particular, the upper row is a split exact sequence and in the
lower row we have $\bbZ\subseteq \ker\det$. Since
$\nu([A])=[\tgA]$ the first square commutes. This implies a
surjective homomorphism $\bar{\nu}$ making the second square
commutative. It remains to show that $\bar{\nu}$ is injective.
Mapping the isomorphism class of a rank $1$ module $P$ to the
symbol $[P]$ in $K_0$ gives a set-theoretical section to both
determinant maps. It suffices to see that $\nu$ preserves the
image of these sections. So let $P$ be a finitely generated
projective $A$-module of rank $1$ and let $[P]\in K_0(A)$. Since
$A$ is an integral domain, $P$ is isomorphic to an invertible
ideal $J\subseteq A$. Consider the invertible ideal $$I=A[X^{\pm
1}]\otimes_A J=A[X^{\pm 1}]J\subseteq A[X^{\pm 1}]$$ and put
$I_0:=I\cap\RA$. Let $P'':=(I_0)^{**}$ be the bidual of the
$\RA$-module $I_0$. By the proof of the previous proposition, it
is a reflexive $R(A)_\bullet$-module in $\Pic(\RA)$ that localizes
to $I$. More precisely, $P''$ generates $I$ over $A[X^{\pm 1}]$ by
means of the duality isomorphism $I\car I^{**}$. Recall that
$R(A)_\bullet$ is noetherian. Hence, $I_0\subseteq\RA$ with its
induced gradation is a finitely generated graded $\RA$-module and
therefore $(I_0)^{*}$ has its natural $\bbZ$-gradation by degree
of morphisms \cite[Cor. 2.4.4]{NastasescuVO}. Since also
$(I_0)^{*}$ is finitely generated, a second application of
loc.cit. implies that $P''=(I_0)^{**}$ has a natural
$\bbZ$-gradation. Similarly, using that $A[X^{\pm 1}]$ is
noetherian, the module $I^{**}$ has a natural $\bbZ$-gradation. It
makes the duality isomorphism $I\car I^{**}$ into a graded
morphism of degree zero (loc.cit., Prop. 2.4.9). It follows that
the isomorphism $(P'')_{(X)}\simeq I$ is a graded isomorphism of
degree zero. One therefore has $[P'']\in K_{0g}(\RA)$ with
$$\beta([P''])=[J]=[P].$$ Consequently, $\nu([P])=[P''/XP''].$
Since $P''/XP''\in \Pic(\tgA)$, this completes the proof.

\end{proof}

\section{Picard groups of forms}

We keep the assumptions on the discretely valued field $k$. We
will work relatively to a finite wildly ramified extension
$k\subseteq\ell.$ Let $p^n=[\tgell:\tgk]$ be the degree of the
purely inseparable extension $\tgk\subseteq\tgell$. Then
$p^n=[\ell:k]$ according to the first remark in subsection
\ref{subsec-gradedreductions}. We restrict to the case of positive
characteristic ${\rm char}~ \tk_1=p>0$. Let $\varpi\in\ell$ be a
uniformizing element, i.e. $|\ell|=|\varpi|^{\bbZ}$ with
$0<|\varpi|<1$. We then have $\tgell \car \tell_1[t^{\pm 1}]$ by
mapping the principal symbol of $\varpi$ to $t$. Under this
identification $\tgk=\tk_1[t^{\pm p^n}].$ For simplicity, we
assume in the following that $\ell/k$ is {\it totally ramified},
i.e. $f=1$ and $\tell_1=\tk_1$. The general case can be treated
similarly but is more technical.

\vskip8pt

By our assumptions the ring extension $\tgk\subseteq\tgell$ is of
the form considered at the end of section \ref{descent}. In
particular, we have the standard higher derivation associated to
it.

\vskip8pt

Let $r\in\sqrt{|k^\times|}$ and $B:=k\{r^{-1}T\}$. Let $A$ be a
form of $B$ with respect to $\ell/k$. We will study the Picard
group of $A$ under the assumption that $A$ has {\it geometrically
reduced graded reduction}, i.e.
$$\tgk^{\rm alg}\otimes_{\tgk} \tgA$$ is reduced where $\tgk^{\rm
alg}$ denotes the graded algebraic closure of $\tgk$. By the usual
argument it suffices to check that
$\tgk^{p^{-1}}\otimes_{\tgk}\tgA$ is reduced.

Suppose now this condition holds. Then Lem. \ref{lem-keyreduced}
implies that $\tgA$ is a form of $\tgB$ for the extension
$\tgk\subseteq\tgell$ which will allow us to use $p$-radical
descent. We fix once and for all an isomorphism
$\tgB\car\tgA\otimes_{\tgk}\tgell$ which we use as an
identification in the following. We now forget all gradations and
treat all graded rings as ordinary rings.

\vskip8pt

Put $m:=p^n-1$. Let $K$ and $K'$ be the (ordinary) fraction fields
of $\tgell$ and $\tgk$. Since $K$ has characteristic $p>0$ we may
consider the standard higher derivation $$\partial: K\rightarrow
K[S]_m$$ of exponent $n$ on $K$ given by $t\rightarrow t+S$ (end
of section \ref{descent}). According to Lem. \ref{lem-hypothesis}
we have $\mu(\De)=1$ and the subring $\tgk$ is invariant. The
field of constants equals $K'$ with $K'\cap \tgell=\tgk$. The
ordered pair $(\tgell,\De)$ satisfies $(\heartsuit)$. Consider the
injective homomorphism $\tgk\rightarrow\tgA$ into the integral
domain $\tgA$. The ring $\tgB=\tgA\otimes_{\tgk} \tgell$ is an
integral domain. Let $L'=\Frac(\tgA)$ be the (ordinary) fraction
field of $\tgA$. Since $\tgB$ is integral over $\tgA$, we then
have
$$L=\Frac (\tgB)=\tgA\otimes_{K'} K$$ for the (ordinary) fraction field of $\tgB$. Now $\tgk$ is a principal ideal domain and, hence, any
torsion free module is flat. Hence, $\tgell$ is flat as a module
over $\tgk$ and the $\tgk$-submodule ${\rm Im}\hskip2pt
(\partial_0-\partial) \subseteq \tgell[S]_m$ associated with the
homomorphism $\partial=\partial_{\tgell}: \tgell\rightarrow
\tgell[S]_m$ is flat. According to Lem. \ref{lem-base} the field
$L$ admits a higher derivation
$$\De_L:={\rm id}_{L'}\otimes_{K'}\partial$$ of the same rank and
order as $\De$ such that the subring $\tgB$ is invariant. The
field of constants with respect to $\partial_L$ equals $L'$ and
$L'\cap \tgB=\tgA$ is a Krull domain. Since $\tgB$ is a Krull
domain, Cor. \ref{cor-baseexp} implies that $(\tgB,\De_L)$
satisfies $(\heartsuit)$. By Prop. \ref{prop-Baba} we therefore
have the group isomorphism

\[\Phi_{\tgB}: \ker\overline{\jmath}\car \cL_{\tgB}/\cL'_{\tgB}, \hskip10pt D{\rm~
mod~}~\dv_{\tgA}(K'^\times)\mapsto \frac{\De(x_D)}{x_D} {\rm
~mod~}\cL'_{\tgB}\] where $x_D\in K^\times$ is chosen such that
$$\dv_{\tgB}(x_D)=j(D)$$ and $$\overline{\jmath}:
\Cl(\tgA)\rightarrow \Cl(\tgB)$$ is the canonical map. Since
$\tgB=\tell_1[t^{\pm},T]$ is a factorial domain \cite[Prop.
VII.\S3.4.3 and Thm. VII.\S3.5.2]{B-CA} we have $\Cl(\tgB)=0$.
Hence

$$\Phi_{\tgB}: \Cl(\tgA)\car\cL_{\tgB}/\cL'_{\tgB}$$
for the class group of the Krull domain $\tgA$.

\begin{lemma}\label{lem-regular}
The ring $\tgA$ is a noetherian regular Krull domain of global
dimension $2$.
\end{lemma}
\begin{proof}
We have already explained above that $\tgA$ is a Krull domain. The ring
extension $\tgA\hookrightarrow\tgB$ makes $\tgA$ a (bi-)module
direct summand of $\tgB$ and $\tgB$ is a free $\tgA$-module via
this extension. By \cite[Thm. 7.2.8(i)]{MCR} we have ${\rm
gld}(\tgA)\leq {\rm gld}(\tgB)=2$, hence $\tgA$ is regular. Its
global dimension coincides then with its Krull dimension and
\cite[Ex.9.2]{Matsumura} implies ${\rm gld}(\tgA)=2$. Since $\tgB$
is noetherian, so is $\tgA$ \cite[Cor. I.\S3.5]{B-CA}.
\end{proof}
If $P\subseteq\tgA$ is a height $1$ prime ideal of $\tgA$, the
localization $(\tgA)_P$ is a discrete valuation ring.  If
$I\subseteq\tgA$ is an ideal, we let $v_P(I)$ be its valuation,
i.e. the integer $\nu$ such that $I_P=P^\nu (\tgA)_P$ as ideals in
$(\tgA)_P$. Since $\tgA$ is a regular Krull domain, \cite[Cor.
I.3.8.1]{WeibelK} says that any divisor $D\in D(\tgA)$ is of the
form $D=\sum_P v_P(I)P$ for a unique invertible ideal $I\subseteq
\tgA$ and that $D\mapsto I$ induces an isomorphism $\Cl(\tgA)\car
\Pic(\tgA)$. Invoking Prop. \ref{prop-comparisonPic} we have
$\bar{\nu}: \Pic(A)\car \Pic(\tgA)$. This proves the following
proposition.

\begin{prop}
There are canonical isomorphisms
$$\Pic(A)\car \Pic(\tgA)\car {\rm \Cl}(\tgA)\car\cL_{\tgB}/\cL'_{\tgB}.$$
In particular, $\Pic(A)$ has exponent $p^n$.
\end{prop}

We give a criterion in which the structure of these groups can be
explicitly determined -at least up to a finite number of choices.
To do this, let us write ${\rm deg}_T$ for the monomial degree in
the variable $T$ of an element in $\tgB=\tilde{\ell}_1[t^{\pm
1},T]$ or of an element in $\tgB[S]_m$. Since $\partial_L(T)=T~
{\rm mod}~ (S)$ we have that ${\rm deg}_T(\partial_{L}(T))\geq 1.$

\begin{prop}
Suppose ${\rm deg}_T(\partial_{L}(T))=1$. Then
$\cL_{\tgB}/\cL'_{\tgB}$ is a finite cyclic group of exponent
$p^n$ generated by the class of $\frac{\partial_L(T)}{T}$.
\end{prop}
\begin{proof}
We use the description of the group $\cL$ from Lem. \ref{lem-log}.
This implies $\frac{\partial_L(T)}{T}\in\cL_{\tgB}$ since
$\partial_L$ leaves $\tgB$ invariant. Moreover, let $f(T)\in
\tgB-\{0\}$ be given such that
$h:=\frac{\partial_L(f)}{f}\in\cL_{\tgB}.$ Then $
\partial_L(f)=f^{\partial}(\partial_L(T))$ where
$f^{\partial}$ refers to the polynomial in $T$ that arises from
$f$ by applying $\partial$ to its coefficients. Our assumption
implies
$${\rm deg}_T(\partial_L(f))={\rm deg}_T (f^{\partial})={\rm deg}_T(f)$$
where the latter identity holds since $\partial$ is injective.
Consequently,
\[\frac{\partial_L(f)}{f}\in (\tgB [S]_m)^\times \Longrightarrow\frac{\partial_L(f)}{f}\in
(\tgell[S]_m)^\times\] and so ${\rm deg}_T(h)=0$. Write
$f=\sum_{i,j}a_{ij}t^{i}T^j$ with $a_{ij}\in \tell_1$ with
$i\in\bbZ,\hskip2pt j\geq 0$. Write $\partial_{L}(T)=Tw$ with some
$w\in\tgell[S]_m$. Thus, $w=\sum_{j=0,...,m}c_jS^j$ with $c_0=1$
and $c_j\in\tgell$ for $j>0$. Now we compute
\[
\sum_{i,j}a_{ij}(t+S)^{i}T^jw^j=\sum_{i,j}a_{ij}\partial(t)^{i}\partial_L(T)^j=
\partial_L(\sum_{i,j}a_{ij}t^{i}T^j)=h \sum_{i,j}a_{ij}t^{i}T^j.\]
Since $f\neq 0$ there is $j$ such that $a_{ij}\neq 0$. Since
$t,S,T$ are algebraically independent and since ${\rm deg}_T(h)=0$
one may compare the coefficient of $T^j$ on both sides which
yields
\[ w^j\,(\sum_{i}a_{ij}(t+S)^{i})/(\sum_{i}a_{ij}t^{i})=h.\]
Since $w=\partial_L(T)/T\in\cL_{\tgB}$ one has
$hw^{-j}\in\cL_{\tgB}$. Using $K\cap\tgB=\tgell$ we obtain
altogether
\[\sum_{i}a_{ij}(t+S)^{i}/(\sum_{i}a_{ij}t^{i})=\partial(\sum_{i}a_{ij}t^{i})/(\sum_{i}a_{ij}t^{i})\in
(K[S]_m)^\times\cap \tgB[S]_m\subseteq \tgell[S]_m.\] In
particular,
$$hw^{-j}=\partial(\sum_{i}a_{ij}t^{i})/(\sum_{i}a_{ij}t^{i})\in
\cL_{\tgell}$$ where $\cL_{\tgell}$ refers to the integral
logarithmic derivatives associated to the standard derivation
$\partial$. Since the latter satisfies $(\heartsuit)$ and since
$\Cl(\tgell)=0$ we have $\cL_{\tgell}=\cL'_{\tgell}$ (Prop.
\ref{prop-Baba}). Since $\partial_L$ extends $\partial$ the
inclusion $\tgell\subseteq\tgB$ induces
$\cL'_{\tgell}\subseteq\cL'_{\tgB}$. Thus, $h\equiv~w^j~{\rm
mod}~\cL'_{\tgB}$ in $\cL_{\tgB}/\cL'_{\tgB}$. Consequently, the
class of $w$ generates the group $\cL_{\tgB}/\cL'_{\tgB}$.
\end{proof}

We conclude by explaining the condition ${\rm
deg}_T(\partial_L(T))=1$ in the context of forms of Russell type.
For simplicity, we let $A$ be a form of $B$ of Russell type with
parameters $n=1$ and $s=r$ (in the notation of Thm.
\ref{thm-russellgraded}) and which admits a trivialization over
$\tgell$. Up to isomorphism, such a form has graded reduction
$$\tgA=\tgk [r^{-p}T_1,r^{-1}T_2]/(T_2^{p}-f(T_1))$$ where
$f(T_1)=a_0T_1+a_1T_1^p+...+a_mT_1^{p^m}\in\tgk[r^{-p}T_1]$ is a
homogeneous $p$-polynomial of degree $r^{p}$ with $a_0\neq 0$. Let
$x$ resp. $y$ be the residue class of $T_1$ and $T_2$ in $\tgA$
respectively. According to subsection
\ref{subsection-gradedformsofRusselltype} we have
$$
1\otimes a_0x= (1\otimes y-(a_m\otimes x^{p^{m-1}}+...+a_1\otimes
x))^p=: t^p$$ in $\theta^{\langle 1 \rangle}A$ with a homogeneous
element $t$ of degree $r$. We have an isomorphism of graded
$\tgell$-algebras
$$ A_{\tgell}\car \tgell [r^{-1}T], \hskip5pt t\mapsto T.$$ Under
this isomorphism $1\otimes y-(a_m^{1/p}\otimes
x^{p^{m-1}}+...+a^{1/p}_1\otimes x)=T$ according to
(\ref{equ-FrobII}). Our condition comes down to the identity
$$\partial(1)\otimes y-(\partial(a_m^{1/p})\otimes
x^{p^{m-1}}+...+\partial(a^{1/p}_1)\otimes
x))=\partial_L(T)=fT+h$$ for elements $f,h\in\tgell[S]_m$. If this
holds, we have
$$ (1-f)\otimes y -([\partial(a_m^{1/p})-fa_m^{1/p}]\otimes
x^{p^{m-1}}+...+[\partial(a^{1/p}_1)-fa_1^{1/p}]\otimes x))=h.$$
The elements $y,x^{p^{m-1}},...,x,1$ are $\tgk$-linearly
independent in $\tgk [x,y]=\tgA$ and therefore $\tgell
[S]_m$-linearly independent in $\tgell [S]_m \otimes_{\tgk}
\tgA=\tgB [S]_m$. It follows $h=0, f=1$ and
$\partial(a_i^{1/p})=a_i^{1/p}$ for all $i=1,...,m$. In other
words, $a_i$ has a $p$-th root in $\tgk$ for $i=1,...,m$.
Conversely, this condition implies $\partial_L(T)=T$. The
preceding two propositions show in this case that $\Pic(A)=1$ and
that $A$ is a principal ideal domain (cf. Lem.
\ref{lem-Dedekind}).

\appendix

\section{Results from graded commutative algebra}
\label{subsection-somegradedcommutativealgebra}

All rings are assumed to be associative, commutative and unital.
Let $\Gamma$ be a commutative multiplicative group.

\subsection{Generalities}

Our reference for general graded rings is \cite{NastasescuVO} but
our exposition follows \cite{ConradTemkin}, \cite{Ducros} and
\cite{TemkinII}. A ring $A$ is called {\it $\Gamma$-graded} or
simply {\it graded} if it allows a direct sum decomposition as
abelian groups
$$A=\oplus_{\gamma\in\Gamma} A_\gamma$$
such that $A_\gamma \cdot A_\delta\subseteq A_{\gamma\delta}$. The
set $\coprod_{\gamma\in\Gamma} A_\gamma$ is the set of {\it
homogeneous} elements. The decomposition of $A$ yields a function
$$\rho_A: \coprod_\gamma (A_\gamma-\{0\})\rightarrow\Gamma,\hskip8pt a\in
A_\gamma-\{0\}\mapsto\gamma,$$ the {\it grading} of $A$. We have
$1\in A_1$ and $A_1$ is a subring of $A$; if $a\in A$ is
homogeneous and invertible in $A$, then $a^{-1}$ is homogeneous
\cite[Prop. 1.1.1]{NastasescuVO}. We denote by $A^\times$ the
group of invertible homogeneous elements in $A$. Then $\rho$
yields a homomorphism $A^\times\rightarrow\Gamma$.

Any ring $A$ may be viewed a graded ring by letting $A_1=A$ and
$A_\gamma=0$ if $\gamma\neq 1$ (the {\it trivial gradation}). A
{\it graded field} is a graded ring $A$ with $A^\times=A-\{0\}$.
Such a ring need not be a field in the usual sense as we will see
below. On the other hand, every field $k$ is a graded field via
the trivial gradation.

\vskip8pt

If $S\subseteq A$ is a subset of homogeneous elements, then the
smallest graded subring of $A$ containing $S$ is called the graded
subring of $A$ {\it generated} by $S$.

\vskip8pt

An ideal $I\subseteq A$ is called {\it graded} if it is generated
by homogeneous elements or, equivalently, if
$I=\oplus_{\gamma\in\Gamma} (I\cap A_\gamma)$. In this case, the
quotient $A/I=\oplus_{\gamma\in\Gamma} A_\gamma/(I\cap A_\gamma)$
is graded.

\vskip8pt

A ring homomorphism $A\rightarrow B$ between graded rings $A$ and
$B$ is {\it graded of degree} $\delta\in\Gamma$ if it maps
$A_\gamma$ into $B_{\delta\gamma}$ for all $\gamma\in\Gamma$. A
homomorphism of degree $1$ will simply be called {\it graded}. In
this case, we call $B$ a {\it graded algebra} over $A$.

\begin{lemma}\label{lem-extension}
Let $k$ be a graded field. Let $A,B$ be two graded $k$-algebras
with $Im ~\rho_A \subseteq Im~ \rho_k$. Then any $k_1$-algebra
homomorphism $A_1\rightarrow B_1$ extends uniquely to a
homomorphism of graded $k$-algebras $A\rightarrow B$.
\end{lemma}
\begin{proof}
Since $A_\gamma= k_\gamma\otimes _{k_1} A_1$ for any
$\gamma\in\Gamma$, this is a simple instance of graded Frobenius
reciprocity \cite[Thm. 2.5.5]{NastasescuVO}. 
\end{proof}

\vskip8pt

A graded ring $A$ is called {\it reduced} if there is no
homogeneous nonzero nilpotent in $A$. A graded ring $A$ is a {\it
graded domain} if all nonzero homogeneous elements of $A$ are not
zero-divisors in $A$. We remark here that a graded ring which is reduced
(resp. a graded domain) need not be reduced (resp. an integral domain)
in the usual sense. However, this is true if $\Gamma$ can be made
into a totally ordered group (e.g. $\Gamma$ is torsion free
\cite{LeviOrdGroups}), in which case one may argue with
homogeneous parts of highest degree.

\vskip8pt

 A {\it prime} homogeneous ideal is a homogeneous ideal
$P\subseteq A$ such that the graded ring $A/P$ is a domain. Such
an ideal need not be a prime ideal in the usual sense. We denote
the set of such ideals by $\Spec_{\Gamma}(A)$ (the {\it graded
prime spectrum} of $A$ \cite[2.11]{NastasescuVO}). Given such
an ideal $P$ we let $A_P$ be the homogeneous localization, i.e.
the localization at the multiplicative subset of the complement
$A-P$ consisting of homogeneous elements. It is a {\it local}
graded ring with a unique maximal homogeneous ideal. The rings
$A_P$ where $P$ runs through $\Spec_{\Gamma}(A)$ are sometimes
called {\it the} local rings of the graded ring $A$. If $A$ is a
graded domain and $P=(0)$, then $A_P$ is a graded field, the {\it
graded fraction field} of $A$. We denote it by ${\rm
Frac}_\Gamma(A)$.

\vskip8pt

A graded domain is called a {\it graded discrete valuation ring}
if it is a graded principal ideal domain (i.e. every homogeneous
ideal admits a single homogeneous generator) with a unique nonzero
prime homogeneous ideal. The proofs of \cite[I.\S2. Prop. 2 and
3]{SerreL} extend to the graded setting by working only with
homogeneous elements. This implies the next lemma (cf. below for a
discussion of the noetherian property and integral elements in the
graded setting).
\begin{lemma}\label{lem-DVR}
For a graded ring $A$ to be a graded discrete valuation ring it is
necessary and sufficient to be a graded noetherian graded local
ring with a maximal ideal generated by a non-nilpotent homogeneous
element. A graded discrete valuation ring is integrally closed. The local rings of a
graded principal ideal domain are graded
discrete valuation rings.
\end{lemma}

Let $A$ be a graded ring. A {\it graded $A$-module} is an
$A$-module $M$ with a direct sum decomposition as abelian groups
$M=\oplus_{\gamma\in\Gamma} M_\gamma$ such that $A_\delta
M_\gamma\subseteq M_{\delta\gamma}$ for all
$\delta,\gamma\in\Gamma$. For the straightforward notions of
graded morphism of degree $\gamma\in\Gamma$, graded submodule,
graded direct sum or graded tensor product we refer to
\cite{NastasescuVO}.

\vskip8pt

If $\gamma\in\Gamma$ we let $A(\gamma)$ be the graded $A$-module
whose underlying $A$-module equals $A$ and whose gradation is
given by $A(\gamma)_\delta=A_{\gamma\delta}$ for all
$\delta\in\Gamma$. If $M$ is a graded $A$-module and $(m_i)_i$ is
a family of elements of $M$ of respective degrees $\gamma_i$,
there exists a unique graded $A$-linear map $\oplus_i
A(\gamma_i^{-1})\rightarrow M$ of degree $1$ mapping $1\in
A(\gamma_i^{-1})$ to $m_i$ for all $i$. The family $(m_i)_i$ is
called {\it free} (resp. {\it generating} resp. {\it a basis}) if
this map is injective (resp. surjective resp. bijective). The
module $M$ is called {\it of finite type} if there exists a
generating family for $M$ of finite cardinality.

\vskip8pt

Let $k$ be a graded field. A graded $k$-module will be called a
{\it graded $k$-vector space}. Any graded $k$-vector space $M$
admits a basis and different bases have the same cardinality. This
cardinality is called the {\it dimension} of $M$ \cite[Lem.
1.2]{TemkinII}.

\vskip8pt

A graded ring is {\it graded noetherian} if any homogeneous ideal
of $A$ is finitely generated. This is equivalent to the property
that chains of homogeneous ideals in $A$ satisfy the ascending
chain condition. In this case, every graded $A$-submodule of a
finitely generated graded $A$-module is finitely generated.

\subsection{Polynomial extensions}

Let $A$ be a graded ring and $r_1,...,r_n\in\Gamma^n$. We let
$$A[r_1^{-1}T_1,...,r_n^{-1}T_n]$$
be the graded ring whose underlying ring equals the polynomial
ring $A[T_1,...,T_n]$ over $A$ and where for all $s\in\Gamma$ the
homogeneous elements of degree $s$ are given by the polynomials
$\sum_I a_I\bT^I$ where $a_I\in A$ is homogeneous of degree
$s\br^{-I}$. Here, $\bT^{I}=T_1^{i_1}\cdot\cdot\cdot T_n^{i_n}$
and $\br^{-I}=(r_1^{-i_1},...,r_n^{- i_n})\in\Gamma^n$ for
$I=(i_1,...,i_n)\in\bbN^n$. In particular, $T_i$ is homogeneous of
degree $r_i$. One has
$A[r_1^{-1}T_1,...,r_n^{-1}T_n]^\times=A^\times$.

If we want to refer to the usual degree in the variable $T$ of the
polynomial underlying an element $f$ of $A[r^{-1}T]$ we will talk
about the {\it monomial degree} of $f$. The proof of the Hilbert
Basis theorem carries over, so $A[r_1^{-1}T_1,...,r_n^{-1}T_n]$ is
graded noetherian if $A$ is graded noetherian
\cite[\S4]{ConradTemkin}.

\vskip8pt

Let $A$ be a graded domain. A nonzero homogeneous element $a\in
A-A^\times$ is called {\it irreducible} if $a=bc$ with homogeneous
elements $b,c\in A$ implies that $b\in A^\times$ or $c\in
A^\times$. The graded domain $A$ is called {\it factorial} if
every nonzero homogeneous element from $A-A^\times$ is uniquely
-up to rearrangement and up to an element from $A^\times$- a
product of irreducible homogeneous elements in $A$. By the
classical argument, a graded principal ideal domain is factorial.
\begin{lemma}\label{lem-gauss}Let $A$ be a graded domain which is
factorial. Let $r\in\Gamma$. Then $A[r^{-1}T]$ is factorial. A
monic homogeneous element in $A[r^{-1}T]$ is irreducible if and
only if it is irreducible in ${\rm Frac}_\Gamma(A)[r^{-1}T]$.
\end{lemma}
\begin{proof}
One has a graded version of the classical Gauss lemma
\cite[Thm. IV.2.1]{LangAlgebra} working only with nonzero
homogeneous elements. One may then copy the arguments from the
ungraded case \cite[Thm. IV.2.3]{LangAlgebra}.
\end{proof}
If $k$ is a graded field, then $k[r^{-1}T]$ is even a principal
ideal domain by \cite[Lemma 4.1]{ConradTemkin}.

\vskip8pt

Let $A\rightarrow B$ be an injective graded homomorphism between
graded rings. If a homogeneous element $b\in B_r$ satisfies
$f(b)=0$ for some monic polynomial $f\in A[T]$ of positive
monomial degree, then one may replace nonzero coefficients of $f$
with suitable nonzero homogeneous parts (depending on $\gamma$) to
find such an $f$ which is homogeneous in $A[r^{-1}T]$. In this
case we call $b$ {\it integral} over $A$. The set of elements of
$B$ whose homogeneous parts are integral over $A$ forms a graded
$A$-subalgebra $A'$ of $B$, the {\it graded integral closure} of
$A$ in $B$. If $A=A'$ we call $A$ {\it integrally closed} in $B$.
If in this situation $A$ is a graded domain and $B$ its graded
fraction field, we call $A$ {\it integrally closed}.

The usual argument with Zorn's lemma using only homogeneous ideals
implies that any homogeneous element $a\in A\setminus A^\times$
lies in a maximal homogeneous ideal. 
This implies $A=\cap_{P\in \Spec_\Gamma(A)} A_P$ for a graded
domain $A$. In this case, $A$ is integrally closed if and only if
this holds for each $A_P$.

\subsection{Graded fields}

Let $k$ be a graded field. Let $p$ be the characteristic exponent 
of the field $k_1$.
\vskip8pt

We begin with an example. Let $\rho$ be the grading of $k$ and let
$r\in\Gamma$ whose class in $\Gamma/\rho(k^\times)$ has infinite
order. The homogeneous elements in $k[r^{-1}T]$ are then given by
the monomials $aT^j$ with $a\in k$ homogeneous and $j\in\bbN$:
indeed, a second monomial $bT^j$ of the same degree leads to
$\rho(a)r^j=\rho(b)r^k$ which implies $j=i$ by the choice of $r$. 
The graded fraction field of $k[r^{-1}T]$ is therefore obtained by
(homogeneous) localization at the single element $T$. We therefore
have
\begin{numequation}\label{equ-gradedfractionfield} k[r^{-1}T,rT^{-1}]:={\rm
Frac}_\Gamma(k[r^{-1}T])=k[r^{-1}T,rS]/(TS-1)\end{numequation} in
case $r$ has infinite order in $\Gamma/\rho(k^\times)$. The
underlying ring of $k[r^{-1}T,rT^{-1}]$ equals the Laurent
polynomials over $k$ in the variable $T$ and so it is not a field
in the usual sense. The homogeneous elements of
$k[r^{-1}T,rT^{-1}]$ are given by $aT^j$ with $a\in k$ homogeneous
and $j\in\bbZ$.

\vskip8pt

A graded homomorphism $k\rightarrow\ell$ between graded fields is
injective and called an {\it extension of graded fields}. If
$x\in\ell^\times$ of degree $r\in\Gamma$, then the graded
homomorphism $$k[r^{-1}T]\rightarrow\ell, \hskip8pt T\mapsto x$$
has a homogeneous kernel $I_x$. If $I_x=0$ resp. $I_x\neq 0$ the
element $x$ is called {\it transcendental} resp. {\it algebraic}
over $k$. The latter case always occurs in case of a finite
extension $k\subseteq\ell$. In the algebraic case we call the
unique monic homogeneous generator $f_x$ of $I_x$ the {\it minimal
homogeneous polynomial} of $x$ over $k$.

The extension $k\subseteq\ell$ is called {\it normal} if $f_x$
splits into a product of homogeneous polynomials in
$\ell[r^{-1}T]$ of monomial degree $1$ for every $x\in\ell^\times$
where $r=\rho_\ell(x)$. The classical argument involving Zorn's
lemma applies to the graded setting and yields a {\it graded
algebraic closure} $\bar{k}$ of $k$, unique up to $k$-isomorphism.
If $f_x$ has only simple roots in $\bar{k}$ then $f_x$ and $x$ are
called {\it separable}. If $f_x$ is of the form $f_x(T)=T^{p^n}-c$
with $c\in k^\times$, we call $f_x$ purely inseparable and $x$
purely inseparable of degree $n$ over $k$. The union $k^s$ of all
separable elements of $\bar{k}$ (together with the zero element)
is called the {\it graded separable closure of $k$ in $\bar{k}$}.
The union $k^{i}$ of all purely inseparable elements (together
with the zero element) is called the {\it graded purely
inseparable closure of $k$ in $\bar{k}$}. Both are graded fields
and one has $k^s\cap k^{i}=k$.

\vskip8pt

 A finite extension $k\subseteq\ell$ which is normal and
separable is called {\it Galois}. In this case, the group of
graded automorphisms of $\ell$ fixing $k$ pointwise is denoted by
$\Gal(\ell/k)$. There is a graded version of the main theorem of
Galois theory \cite[(1.16.1)]{Ducros}. In particular,
$\#\Gal(\ell/k)=[\ell:k]$ and $\ell^{\Gal(\ell/k)}=k$.


\vskip8pt

\vskip8pt

Let $k\subseteq\ell$ be an extension of graded fields and
$S\subseteq\ell$ a subset. We let $k(S)$ be the smallest graded
subfield of $\ell$ that contains $k$ and $S$. It equals the graded
fraction field of the graded subring of $\ell$ generated by $k$
and $S$. If $\ell=k(S)$ we say that $\ell$ is {\it generated by
$S$ over $k$}.

\vskip8pt

Let $k\subseteq\ell$ be an extension of graded fields. A subset
$S\subseteq\ell^\times$ is said to be {\it algebraically
independent over $k$} if the graded $k$-algebra homomorphism
$k[r_1^{-1}T_1,...,r_n^{-1}T_n]\rightarrow\ell$ defined by
$T_i\mapsto s_i$ is injective for any finite subset $s_1,...,s_n$
of $S$ (with $r_i=\rho_\ell(s_i)$). A subset $S$ which is a
maximal algebraically independent set over $k$ is called a {\it
transcendence basis for $\ell/k$}. As in the classical case, one
shows that a transcendence basis always exists and that all such
bases have the same cardinality, the {\it transcendence degree}
${\rm trdeg}_k (\ell)$ of the extension $\ell/k$
\cite[\S4]{ConradTemkin}. Similarly, if the graded field $\ell$
is generated over $k$ by a set $S\subseteq\ell$, one may
choose a transcendence basis from $S$. 
We shall need a graded version on the existence of a {\it
separating} transcendence basis in the following sense.
\begin{lemma}\label{lem-septransbasis}
Suppose $k\subseteq\ell$ is an extension of graded fields such that
the graded ring $k^{i}\otimes_k\ell$ is reduced. If $S\subseteq
\ell$ generates $\ell$ over $k$, there is a subset $S'\subseteq S$
with the properties: {\rm (i)} $S'$ is a transcendence basis of
$\ell/k$, {\rm (ii)} the algebraic extension $k(S')\subseteq\ell$
is separable.
\end{lemma}
\begin{proof}
Let $r$ and $r_1,...,r_n$ be elements of $\Gamma$. The graded ring
$k[r_i^{-1}T]$ is a principal ideal domain and therefore
factorial. Hence $A:=k[r_1^{-1}T_1,...,r_n^{-1}T_n]$ is factorial
according to Lem. \ref{lem-gauss}. The latter lemma also implies
that a monic homogeneous polynomial in $A[r^{-1}T]$ is irreducible
if and only if it is irreducible in ${\rm
Frac}_\Gamma(A)[r^{-1}T]$. We then have everything to extend the
classical proof \cite[Prop. VIII.4.1]{LangAlgebra} on the
existence of a separating transcendence basis for separable field
extensions to the graded setting by working with homogeneous
elements only.
\end{proof}
\begin{lemma}\label{lem-trdeg}
If $K$ denotes the graded fraction field of $k[r^{-1}T]$, then
${\rm trdeg}_k(K)=1$.
\end{lemma}
\begin{proof}
According to \cite[Lemma 4.8]{ConradTemkin} we have the formula
$$ {\rm trdeg}_k(K)={\rm trdeg}_{k_1}(K_1)+{\rm
dim}_{\bbQ}((\rho(K^\times)/\rho(k^\times))\otimes_{\bbZ}\bbQ).$$
Suppose first that the image of $r$ has finite order in
$\Gamma/\rho(k^\times)$. There is a finite extension $k\subseteq
\ell$ with $r=\rho(a)$ for some $a\in\ell^\times$. We then have an
isomorphism $\ell[r^{-1}T]\car \ell[T]$ of graded $\ell$-algebras
by mapping $T\mapsto aT$. Since ${\rm trdeg}_k(K)={\rm trdeg}_\ell
(K(\ell))$ and since $K(\ell)={\rm Frac}_\Gamma(\ell[r^{-1}T])$,
we may therefore assume $r=1$. But then $K_1=k_1(T)$, the usual
field of rational functions over $k_1$ and we are done. In the
case, where the image of $r$ has infinite order in
$\Gamma/\rho(k^\times)$, we have $K=k[r^{-1}T,rT^{-1}]$ by
(\ref{equ-gradedfractionfield}). Hence $K_1=k_1$. Since
$(\rho(K^\times)/\rho(k^\times))\otimes_{\bbZ}\bbQ$ is the line
generated by the image of $r$, the assertion follows in this case,
too.
\end{proof}

\vskip8pt

\bibliographystyle{plain}
\bibliography{mybib}

\begin{thebibliography}{10}

\bibitem{AuslanderBuchsbaum}
M.~Auslander and D.~A. Buchsbaum.
\newblock Homological dimension in local rings.
\newblock {\em Trans. Amer. Math. Soc.}, 85:390--405, 1957.

\bibitem{AuslanderBuchsbaumII}
M.~Auslander and D.~A. Buchsbaum.
\newblock Unique factorization in regular local rings.
\newblock {\em Proc. Nat. Acad. Sci. U.S.A.}, 45:733--734, 1959.

\bibitem{Baba}
K.~Baba.
\newblock On {$p$}-radical descent of higher exponent.
\newblock {\em Osaka J. Math.}, 18(3):725--748, 1981.

\bibitem{Bass}
H.~Bass.
\newblock {\em Algebraic {$K$}-theory}.
\newblock W. A. Benjamin, Inc., New York-Amsterdam, 1968.

\bibitem{Berkovichbook}
Vladimir~G. Berkovich.
\newblock {\em Spectral theory and analytic geometry over non-archimedean
  fields}, volume~33 of {\em Math. {S}urveys and {M}onographs}.
\newblock American {M}athematical {S}ociety, Providence, {R}hode {I}sland,
  1990.

\bibitem{BGR}
S.~Bosch, U.~G{\"u}ntzer, and R.~Remmert.
\newblock {\em Non-{A}rchimedean analysis}.
\newblock Springer-Verlag, Berlin, 1984.

\bibitem{B-CA}
N.~Bourbaki.
\newblock {\em Commutative algebra. {C}hapters 1--7}.
\newblock Elements of Mathematics (Berlin). Springer-Verlag, Berlin, 1998.

\bibitem{ConradTemkin}
B.~Conrad and M.~Temkin.
\newblock Descent for non-archimedean analytic spaces.
\newblock {\em Preprint 2010. {A}vailable at:
  {$http://math.huji.ac.il/~temkin/papers/Descent.pdf$}}.

\bibitem{DemazureGabriel}
M.~Demazure and P.~Gabriel.
\newblock {\em Groupes alg\'ebriques. {T}ome {I}: {G}\'eom\'etrie alg\'ebrique,
  g\'en\'eralit\'es, groupes commutatifs}.
\newblock Masson \& Cie, \'Editeur, Paris, 1970.
\newblock Avec un appendice {{\i}t Corps de classes local} par Michiel
  Hazewinkel.

\bibitem{Ducros}
A.~Ducros.
\newblock Toute forme mod\'er\'ement ramifi\'ee d'un polydisque ouvert est
  triviale.
\newblock {\em Math. Z.}, 273(1-2):331--353, 2013.

\bibitem{Jacobson}
N.~Jacobson.
\newblock {\em Lectures in abstract algebra. {III}}.
\newblock Springer-Verlag, New York, 1975.
\newblock Theory of fields and Galois theory, Graduate Texts in Math., No. 32.

\bibitem{KMT}
T.~Kambayashi, M.~Miyanishi, and M.~Takeuchi.
\newblock {\em Unipotent algebraic groups}.
\newblock Lecture Notes in Mathematics, Vol. 414. Springer-Verlag, Berlin,
  1974.

\bibitem{Kaplansky}
I.~Kaplansky.
\newblock Maximal fields with valuations.
\newblock {\em Duke Math. J.}, 9:303--321, 1942.

\bibitem{KnusOjanguren}
M.-A. Knus and M.~Ojanguren.
\newblock {\em Th\'eorie de la descente et alg\`ebres d'{A}zumaya}.
\newblock Lecture Notes in Math., Vol. 389. Springer-Verlag, Berlin, 1974.

\bibitem{LangAlgebra}
S.~Lang.
\newblock {\em Algebra}, volume 211 of {\em Graduate Texts in Mathematics}.
\newblock Springer-Verlag, New York, third edition, 2002.

\bibitem{LeviOrdGroups}
F.~W. Levi.
\newblock Ordered groups.
\newblock {\em Proc. Indian Acad. Sci., Sect. A.}, 16:256--263, 1942.

\bibitem{LVOVB}
H.~Li, M.~Van~den Bergh, and F.~Van~Oystaeyen.
\newblock Note on the {$K\sb 0$} of rings with {Z}ariskian filtration.
\newblock {\em $K$-Theory}, 3(6):603--606, 1990.

\bibitem{LVOGlobalAuslanderRees}
H.~Li and F.~van Oystaeyen.
\newblock Global dimension and {A}uslander regularity of {R}ees rings.
\newblock {\em Bull. {M}ath. {S}oc. {B}elgique}, (serie A) XLIII:59--87, 1991.

\bibitem{LVO}
H.~Li and F.~van Oystaeyen.
\newblock {\em Zariskian filtrations}, volume~2 of {\em $K$-Monographs in
  Mathematics}.
\newblock Kluwer Academic Publishers, Dordrecht, 1996.

\bibitem{Matsumura}
H.~Matsumura.
\newblock {\em Commutative ring theory}, volume~8 of {\em Cambridge Studies in
  Advanced Mathematics}.
\newblock Cambridge University Press, Cambridge, 1986.

\bibitem{MCR}
J.~C. McConnell and J.~C. Robson.
\newblock {\em Noncommutative {N}oetherian rings}.
\newblock Pure and Applied Mathematics (New York). John Wiley \& Sons Ltd.,
  Chichester, 1987.

\bibitem{NastasescuVO}
C.~N{\u{a}}st{\u{a}}sescu and F.~Van~Oystaeyen.
\newblock {\em Methods of graded rings}, volume 1836 of {\em Lecture Notes in
  Mathematics}.
\newblock Springer-Verlag, Berlin, 2004.

\bibitem{QuillenH}
D.~Quillen.
\newblock Higher algebraic {$K$}-theory. {I}.
\newblock In {\em Algebraic {$K$}-theory, {I}: {H}igher {$K$}-theories ({P}roc.
  {C}onf., {B}attelle {M}emorial {I}nst., {S}eattle, {W}ash., 1972)}, pages
  85--147. Lecture Notes in Math., Vol. 341. Springer, Berlin, 1973.

\bibitem{RemyThuillierWerner10}
B.~R{\'e}my, A.~Thuillier, and A.~Werner.
\newblock Bruhat-{T}its theory from {B}erkovich's point of view. {I}.
  {R}ealizations and compactifications of buildings.
\newblock {\em Ann. Sci. \'Ec. Norm. Sup\'er. (4)}, 43(3):461--554, 2010.

\bibitem{Russell}
P.~Russell.
\newblock Forms of the affine line and its additive group.
\newblock {\em Pacific J. Math.}, 32:527--539, 1970.

\bibitem{Samuel2}
P.~Samuel.
\newblock Classes de diviseurs et d\'eriv\'ees logarithmiques.
\newblock {\em Topology}, 3(suppl. 1):81--96, 1964.

\bibitem{SerreL}
J.-P. Serre.
\newblock {\em Local fields}, volume~67 of {\em Graduate Texts in Math.}
\newblock Springer-Verlag, New York, 1979.

\bibitem{SerreG}
J.-P. Serre.
\newblock {\em Galois cohomology}.
\newblock Springer Monographs in Math. Springer-Verlag, Berlin, english
  edition, 2002.

\bibitem{TemkinI}
M.~Temkin.
\newblock On local properties of non-{A}rchimedean analytic spaces.
\newblock {\em Math. Ann.}, 318(3):585--607, 2000.

\bibitem{TemkinII}
M.~Temkin.
\newblock On local properties of non-{A}rchimedean analytic spaces. {II}.
\newblock {\em Israel J. Math.}, 140:1--27, 2004.

\bibitem{Waterhouse}
W.~C. Waterhouse.
\newblock {\em Introduction to affine group schemes}, volume~66 of {\em
  Graduate Texts in Mathematics}.
\newblock Springer-Verlag, New York, 1979.

\bibitem{WeibelK}
C.~Weibel.
\newblock An introduction to algebraic {K}-theory.
\newblock {\em Preprint. Available at:
  {$http://www.math.rutgers.edu/~weibel/Kbook.html$}}.

\end{thebibliography}

\end{document}